%% file: schreck_etal_2014_HAL.tex
\documentclass[final]{siamltex1213}

\usepackage{algorithmic,ifthen}
\usepackage[export]{adjustbox}
\usepackage{enumitem}
\usepackage{amsmath}
 \usepackage{graphicx}
 \usepackage{marvosym}
 \usepackage{amssymb}
 \usepackage{amsbsy}
 \usepackage{enumitem}
 \usepackage{bbm}
\usepackage[lined,boxed,linesnumbered]{algorithm2e}
\usepackage{xargs}

\usepackage{aliascnt}
\usepackage{ntheorem,cleveref}

\renewtheorem{theorem}{Theorem}[section]
\crefname{theorem}{theorem}{Theorems}
\Crefname{Theorem}{Theorem}{Theorems}

\newaliascnt{lemma}{theorem}
\renewtheorem{lemma}[theorem]{Lemma}
\aliascntresetthe{lemma}
\crefname{lemma}{lemma}{lemmas}
\Crefname{Lemma}{Lemma}{Lemmas}

\newaliascnt{corollary}{theorem}
\renewtheorem{corollary}[theorem]{Corollary}
\aliascntresetthe{corollary}
\crefname{corollary}{corollary}{corollaries}
\Crefname{Corollary}{Corollary}{Corollaries}

\newaliascnt{proposition}{theorem}
\renewtheorem{proposition}[theorem]{Proposition}
\aliascntresetthe{proposition}
\crefname{proposition}{proposition}{propositions}
\Crefname{Proposition}{Proposition}{Propositions}

\crefname{figure}{figure}{figures}
\Crefname{Figure}{Figure}{Figures}

\newaliascnt{definition}{theorem}
\renewtheorem{definition}[theorem]{Definition}
\aliascntresetthe{definition}
\crefname{definition}{definition}{definitions}
\Crefname{Definition}{Definition}{Definitions}

\newtheorem{assumption}{\textbf{H}\hspace{-3pt}}
\Crefname{assumption}{H\hspace{-3pt}}{H\hspace{-3pt}}
\crefname{assumption}{H\hspace{-3pt}}{H\hspace{-3pt}}

\newtheorem{assumptionE}{\textbf{E}\hspace{-3pt}}
\Crefname{assumptionE}{E\hspace{-3pt}}{E\hspace{-3pt}}
\crefname{assumptionE}{E\hspace{-3pt}}{E\hspace{-3pt}}

\newcommand{\coint}[1]{\left[#1\right)}
\newcommand{\ocint}[1]{\left(#1\right]}
\newcommand{\ooint}[1]{\left(#1\right)}
\newcommand{\ccint}[1]{\left[#1\right]}

\def\pa{\upsilon}
\def\bu{\mathsf{a}}

\def\calH{\mathcal{H}}

\def\calC{\mathcal{C}}

\def\calW{\mathcal{W}}

\def\calL{\mathcal{L}}
\def\calA{\mathcal{A}}
\def\calB{\mathcal{B}}
\def\calK{\mathcal{K}}
\def\calP{\mathcal{P}}

\def\A{A}

\newcommand{\bs}{\boldsymbol}

\newcommand{\eqdef}{\ensuremath{\stackrel{\mathrm{def}}{=}}}
\def\eqsp{} 
\def\as{\ensuremath{\mathrm{a.s.}}}
\newcommand{\F}{\mathcal{F}}
\newcommand{\PE}{\mathbb{E}}
\newcommand{\PP}{\mathbb{P}}
\newcommand{\X}{\mathsf{X}}
\newcommand{\Xfield}{\mathcal{X}}
\newcommand{\thetaset}{\Theta}
\newcommand{\cgeom}{\lambda}

\newcommand{\rmd}{\mathrm{d}}

\def\ints{\mathrm{int}}
\def\mcf{\mathcal{F}}

\newcommand{\pas}{\gamma}

\newcommand{\dimpar}{d}
\def\Rset{\mathbb{R}}
\def\rset{\mathbb{R}}
\def\Nset{\mathbb{N}}

\def\un{\mathbbm{1}}
\newcommand{\indic}[1]{\mathbbm{1}_{\{#1\}}}
\newcommand{\param}{\theta}
\newcommand{\btheta}{{\bf \theta}}

\newcommand{\quant}{q}
\newcommand{\incr}{t}
\DeclareMathOperator*{\argmin}{arg\,min}

\newcommand{\norm}[1]{\left| #1 \right|}
 \newcommand{\pscal}[2]{\left< #1, #2
  \right>} \newcommand{\suite}[1] {\left\{ #1_n, n \in \Nset \right\}}

\renewcommand{\H}[2]{\ifthenelse{\equal{#2}{}}{H_{#1}}{H_{#1}(#2)}}
\newcommand{\tH}[2]{\ifthenelse{\equal{#2}{}}{\widetilde{H_{#1}}}{\widetilde{H_{#1}}(#2)}}

\renewcommand{\P}[1]{P_{\param_{#1}}}
\newcommand{\g}[1]{g_{\param_{#1}}}

\def\ie{\emph{i.e.}}

\def\eg{e.g.}
\def\iid{i.i.d.}
\def\wrt{w.r.t.}

\newcommandx\sequence[3][2=,3=]
{\ifthenelse{\equal{#3}{}}{\ensuremath{\{ #1_{#2}\}}}{\ensuremath{\{ #1_{#2}, \eqsp #2 \in #3 \}}}}
\def\argmax{\mathrm{argmax}}
\begin{document}
\title{Convergence of Markovian Stochastic Approximation with discontinuous dynamics}
\author{  G. Fort \footnotemark[1] \footnotemark[2]
\and E. Moulines \footnotemark[1]
\and A. Schreck \ \footnotemark[1]
\and M. Vihola \footnotemark[3]}

\maketitle

\slugger{sicon}{xxxx}{xx}{x}{x--x}

\renewcommand{\thefootnote}{\fnsymbol{footnote}}

\footnotetext[2]{corresponding author. mail: gersende.fort@telecom-paristech.fr}
\footnotetext[1]{LTCI ; T\'el\'ecom ParisTech \& CNRS}
\footnotetext[3]{University of Jyv\"askyl\"a ; Department of Mathematics and Statistics}

\renewcommand{\thefootnote}{\arabic{footnote}}

\begin{abstract}
  This paper is devoted to the convergence analysis of stochastic approximation
  algorithms of the form $\theta_{n+1} = \theta_n + \gamma_{n+1}
  \H{\theta_n}{X_{n+1}}$ where $\suite{\theta}$ is a $\Rset^d$-valued sequence,
  $\suite{\gamma}$ is a deterministic step-size sequence and $\suite{X}$ is a
  controlled Markov chain. We study the convergence under weak assumptions on
  smoothness-in-$\theta$ of the function $\theta \mapsto \H{\theta}{x}$. It is
  usually assumed that this function is continuous for any $x$; in this work,
  we relax this condition. Our results are illustrated by considering
  stochastic approximation algorithms for (adaptive) quantile estimation and a
  penalized version of the vector quantization.
\end{abstract}



\pagestyle{myheadings}
\thispagestyle{plain}
\markboth{G. FORT ET~AL.}{CONVERGENCE OF DISCONTINUOUS MARKOVIAN SA}

\section{Introduction}

Stochastic Approximation (SA) methods have been introduced by
\cite{robbins:monro:1951} as algorithms to find the roots of
$h:\thetaset \to \Rset^\dimpar$ where $\thetaset$ is an open subset of
$\Rset^{\dimpar}$ (equipped with its Borel $\sigma$-field
$\mathcal{B}(\thetaset)$) when only noisy measurements of $h$ are
available. More precisely, let $\X$ be a space equipped with a
countably generated $\sigma$-field $\Xfield$, $\{P_{\param}, \param
\in \thetaset \}$ be a family of transition kernels on $(\X, \Xfield)$
and $H: \X \times \thetaset \to \Rset^{\dimpar}$, $(x,\param) \mapsto
\H{\param}{x}$ be a measurable function. We consider
\begin{equation}
  \label{eq:AS:generaldefinition}
  \param_{n+1} = \param_n + \pas_{n+1} \H{\param_n}{X_{n+1}}
\end{equation}
where $\suite{\pas}$ is a sequence of deterministic nonnegative step sizes and
$\suite{X}$  is a controlled  Markov chain, \ie, for any non-negative measurable function $f$,
\[
\PE[f(X_{n+1})\vert \mcf_n]= P_{\param_n} f(X_{n}) \eqsp, \quad \PP-\as\  \eqsp, \quad \mcf_n= \sigma( (X_\ell,\param_\ell), \ell \leq n) \eqsp.
\]
It is assumed that for each $\param \in \thetaset$, $P_\param$ admits a unique
stationary distribution $\pi_\param$ and that $h(\param)= \int_{\X}
\H{\param}{x} \pi_\param(\rmd x) = \pi_\param(\H{\param}{})$ (assuming that
$\pi_\param(|H_\param|) < \infty$).  This setting encompasses the cases
$\suite{X}$ is a (non-controlled) Markov chain by choosing $P_\param =P$ for
any $\param$; the Robbins-Monro case by choosing $P_\param(x,\cdot) =
\pi_\param(\cdot)$ where $\pi_\param$ is a distribution on $\X$; the case when
$\suite{X}$ is an i.i.d.  sequence with distribution $\pi$ by choosing
$P_\param(x,\cdot) = \pi(\cdot)$ for any $x,\param$.

The goal of this paper is to provide almost sure convergence results of the
sequence $\suite{\param}$ under conditions on the regularity of the $\param
\mapsto \H{\param}{x}$ which does not include continuity, which is usually
assumed in the literature.  When $\suite{X}$ is a controlled Markov chain,
\cite{benveniste:metivier:priouret:1990} and \cite[Theorem 4.1]{tadic:1998}
establish \as\ convergence under the assumption that for any $x$, $\param
\mapsto \H{\param}{x}$ is H\"older-continuous. This assumption traces back to
\cite[Eq.  (4.2)]{kushner:1981} and the same assumption is imposed in
\cite[assumption (DRI2) and Proposition~6.1.]{andrieu:moulines:priouret:2005}.

In order to prove convergence, a preliminary step is to establish that the
sequence $\suite{\param}$ is $\PP$-\as\ in a compact set of $\Theta$, a
property referred to as \emph{stability} in~\cite{kushner:yin:2003}. It is
common in applications that stability fails to hold or it cannot be
theoretically guaranteed. When the `unconstrained' process as stated above can
be shown to be stable, the proof often requires problem specific arguments; see
for instance \cite{saksman:vihola:2010,andrieu:tadic:vihola:2015}.  Different
algorithmic modifications for ensuring stability have been suggested in the
literature.  It is sometimes possible to modify $H$ without modifying the
stationary points in order to ensure stability as suggested in
\cite{lemaire:pages:2010} (see also \cite{bardou:frikha:pages:2009} and
\cite{jourdain:lelong:2009} for applications of this approach). An alternative
is to adapt the step sizes that control the growth of the iterates
(\cite{kamal:2010}).  Another idea is to replace the single draw in
(\ref{eq:AS:generaldefinition}) by a Monte Carlo sum over many realizations of
$X_{n+1}$ (\cite{younes:1999}). Such modifications usually require quite
precise understanding of the properties of the system in order to be
implemented efficiently.  The values of $\theta_n$ may simply be constrained to
take values in a compact set $K$ \cite{kushner:clark:1978,kushner:yin:2003}.
The choice of the constraint set $K$ requires prior information about the
stationary points of $h$, as ill-chosen constraint set $K$ may even lead to
spurious convergence on the boundary of $K$.  It is possible to modify the
projection approach by constraining $\theta_n$ to take values in compact sets
$K_n$, which eventually cover the whole parameter space $\cup_n K_n =
\thetaset$ \cite{andradottir:1995}. In the controlled Markov chain setup, this
approach requires relatively good control on the ergodic properties of the
related Markov kernels near the `boundary' of $\thetaset$
\cite{andrieu:vihola:2011}.

We focus on the self-stabilized stochastic approximation algorithm
with controlled Markov chains \cite{andrieu:moulines:priouret:2005},
which is based on truncations on random varying sets as suggested in
\cite{chen:zhu:1986}. The main difference to the expanding projections
approach of \cite{andradottir:1995,andrieu:vihola:2011} is the occasional 
`restart' of the process; see Section \ref{section:cv:results}.  
The main advantages of this algorithm include 
that it does not introduce spurious convergence as projections to
fixed set, it provides automatic calibration of the step size, but
it does not require precise control of the behavior of the system near
the boundary of $\thetaset$ like the expanding projections and the
averaging approaches. The convergence properties of the algorithm
under the controlled Markov chain setup is studied in 
\cite{andrieu:moulines:priouret:2005} and in
a different setup in \cite{lelong:2013}. This algorithm has been used 
in various applications, including adaptive Monte Carlo
\cite{lapeyre:lelong:2011} and
adaptive Markov chain Monte Carlo \cite{andrieu:moulines:2006}.

\Cref{theorem:general} provides sufficient conditions implying that the
number of truncations is finite almost surely.  Therefore, this stabilized
algorithm follows the equation (\ref{eq:AS:generaldefinition}) after some
random (but almost surely finite) number of iterations.  We then prove the
almost sure convergence of $\suite{\param}$ to a connected component of a
limiting set which contains the roots of $h$. We also provide a new set of
sufficient conditions for the almost sure convergence of a SA sequence which
weakens the conditions used in earlier contributions (see
\Cref{lemma:rappel:AMP:item2}). We illustrate our results for (adaptive)
quantile and multidimensional median approximation.  We also analyze a
penalized version of the $0$-neighbors Kohonen algorithm.

The paper is organized as follows: the stabilized stochastic approximation
algorithm, the main assumptions and the convergence result are given in
\Cref{sec:maintheorems}.  \Cref{section:examples} is devoted to 
applications.  The proofs are postponed to Sections \ref{section:proof} and
\ref{section:proofs-examples}.

\section{Main results}
\label{section:cv:results}
\label{sec:maintheorems}
\label{sec:description:algo}
\label{sec:assumptions}

Let $\{\calK_q, q \in \Nset\}$ be a
sequence of compact subsets of $\thetaset$ such that
\begin{equation}
\label{eq:ConditionCompact} \bigcup_{i \geq 0} \calK_i= \thetaset,
\quad \textrm{and} \quad \calK_i \subset \ints (\calK_{i+1}), \quad
i \geq 0,
\end{equation}
where $\ints(A)$ denotes the interior of the set $A$.  The stable algorithm,
described in \Cref{algo:SA:stabilise}, proceeds as follows. We first run {\tt
  SA}$(\bs \gamma, \calK_0, x, \param)$ (see \Cref{algo:SA:basique}) until the
first time instant for which $\param_{n} \not \in \calK_0$. When it occurs,
\textit{(i)} the active set is replaced with a larger one $\calK_1$,
\textit{(ii)} the stepsize sequence $\bs \gamma$ is shifted and replaced with
the sequence $\bs \gamma^{\leftarrow 1} \eqdef \{\gamma_{1+k}, k \in \Nset \}$
and \textit{(iii)} {\tt SA}$(\bs \gamma^{\leftarrow 1}, \calK_1, x,\param)$ is
run. The above procedure is repeated until convergence; see \eg\ ~\cite{andrieu:moulines:priouret:2005}).
\begin{algorithm}[h]
  \SetKwInOut{Input}{Input} \Input{A positive sequence $\bs \rho = \{\rho_n, n
    \in \Nset \}$ \; A point $(x,\param) \in \X \times \thetaset$ \;A subset
    $\calK \subseteq \thetaset$} Initialization: $(X_0, \param_0) =
  \left(x,\param \right)$, $n =0$ \; \Repeat{$\param_{n} \notin \calK$}{Draw
    $X_{n+1} \sim P_{\param_n}(X_n, \cdot)$ \label{line:algo1:updateX} \;
    $\param_{n+1} = \param_n + \rho_{n+1} \H{\param_n}{X_{n+1}}$
    \label{line:algo1:updateT}  \; $n \leftarrow n+1$}
\caption{Stochastic Approximation algorithm {\tt SA}$(\bs \rho,\calK,x,\param)$ \label{algo:SA:basique}}
\end{algorithm}

\begin{algorithm}[ht!]
  \SetKwInOut{Input}{Input} \SetKwInOut{Output}{Output} \Input{A positive
    sequence $\bs \gamma = \suite{\gamma}$ and a  $(x,\param) \in \X \times
    \thetaset$; } \Output{The sequence $\{(X_n, \param_n, I_n), n \in \Nset \}$
  } $(X_0, \param_0, I_0) = \left(x,\param, 0 \right)$ \; $n =0$\tcc*{$n$:
    number of iterations} $\zeta_0=0$ \tcc*{$\zeta$: number of iterations in
    the current active set} \Repeat{convergence of the sequence $\{\param_n, n
    \in \Nset \}$}{ \eIf{$\zeta_n=0$
      \label{line:nonhomMCinit:algo2}}{ $X_{n+1/2} = x$, $\param_{n+1/2} = \theta$ \;}{$X_{n+1/2} = X_n$, $\param_{n+1/2} =
        \param_n$ \;} Draw $X_{n+1} \sim P_{\param_{n+1/2}}(X_{n+1/2}, \cdot)$
      \; $\param_{n+1} = \param_{n+1/2} + \gamma_{I_n+\zeta_n+1}
      \H{\param_{n+1/2}}{X_{n+1}}$ \; \eIf{$\param_{n+1} \in \calK_{I_n}$}{$I_{n+1} =
        I_n$, $\zeta_{n+1}= \zeta_n+1$
        \label{line:nonhomMCend:algo2}}{$I_{n+1} = I_n+1$, $\zeta_{n+1}=0$} $n \leftarrow n+1$
 }
\caption{A Stable Stochastic Approximation algorithm \label{algo:SA:stabilise}}
\end{algorithm}
Consider the following assumptions:
\begin{assumption} \label{hyp:PhiandH} The function $(x,\param) \mapsto \H{\param}{x}$  from
  $ \X \times \thetaset$ to $\Rset^d$ is measurable. There exists a measurable
  function $W: \X \to \coint{1, \infty}$ such that for any compact set $\calK
  \subset \thetaset$, $\sup_{\param \in \calK} |\H{\param}{}|_W < \infty$ where
  $|f|_W \eqdef \sup_\X  |f|/W$.
\end{assumption}

\begin{assumption} \label{hyp:noyau}
 \begin{enumerate}[label=(\alph*)]
 \item \label{hyp:noyau:loi_inv} For any $\param$ in $\thetaset$, the kernel
   $P_{\param}$ has a unique invariant distribution $\pi_{\param}$.
 \item \label{hyp:noyau:ergo:geom} For any compact $\mathcal{K}\subseteq
   \thetaset$, there exist constants $C>0$ and $\cgeom \in \ooint{0,1}$ such
   that for any $x \in \X$, $l \geq 0$,
   $\sup_{\param \in \mathcal{K}} \| \P{}^l(x,.)-\pi_{\param} \|_W \leq C
    \cgeom^l W(x)$ and $\sup_{\param \in \calK} \pi_\param(W) < \infty$,
 where $\| \mu \|_W \eqdef \sup_{\{f, |f|_W \leq 1\}} |\mu(f)|$.
\item \label{hyp:noyau:control:W:traj} There exists $p >1$ and for any compact
  set $\calK \subset \thetaset$, there exist constants $\varrho \in
  \ooint{0,1}$ and $b < \infty$ such that
  $\sup_{\param \in \calK}   P_\theta W^p(x) \leq \varrho W^p(x) + b$.
 \end{enumerate}
\end{assumption}
When $\P{} = P$ for any $\param$, sufficient conditions for \Cref{hyp:noyau}-\ref{hyp:noyau:ergo:geom}
are given in \cite[Chapters 10 and 15]{meyn:tweedie:1993}: they are mainly
implied by the drift condition \Cref{hyp:noyau}-\ref{hyp:noyau:control:W:traj} assuming the level sets $\{W \leq M\}$ are petite~\cite[Lemma 2.3]{fort:moulines:priouret:2010}.

We also introduce an assumption on the smoothness-in-$\param$ of the transition
kernels $\{P_\param, \param \in \thetaset \}$. Let $\pscal{\cdot}{\cdot}$
denote the usual scalar product in $\Rset^d$ and $\norm{\cdot}$ denote the
associated norm. Denote
\begin{equation}
  \label{def:D}
  D_W(\param,\param') \eqdef \sup_{x\in \X} \frac{\|P_{\param}(x,\cdot) - P_{\param'}(x,\cdot) \|_W}{W(x)} \eqsp.
\end{equation}
\begin{assumption}\label{hyp:Wfluctuation}
  There exists $\pa \in \ocint{0,1}$ such that for any compact $\calK \subset
  \thetaset$, 
  $$\sup_{\param, \param' \in \calK} \frac{D_W(\param,
\param')}{\norm{\param - \param'}^{\pa}}  < \infty \eqsp.
  $$
\end{assumption}
In \cite[Section 6]{andrieu:moulines:priouret:2005} it is assumed that there
exists $\alpha \in \ocint{0,1}$ such that for any compact set $\calK$,
$\sup_{\param_1,\param_2 \in \calK} \norm{ \param_1 - \param_2 }^{-\alpha} \ 
|\H{\param_1}{}-\H{\param_2}{}|_W < \infty$.  We consider here a weaker
condition.
 \begin{assumption} \label{hyp:H:deltaH} Let $\alpha \in \ocint{0,1}$. For any compact set $\mathcal{K}\subseteq \thetaset$,
   there exists a constant $C>0$ such that for all $\delta>0$,
\begin{align*}
  \sup_{ \param \in \mathcal{K}} \int \pi_{\param}(\rmd x) \sup_{\{\param' \in \calK, \norm{
    \param'- \param} \leq \delta \}} \norm{ \H{\param'}{x} - \H{\param}{x}} \leq C \, \delta^{\alpha} \eqsp.
\end{align*}
\end{assumption}
Denote by $h$ the mean field
  \begin{equation}
    \label{eq:meanfield}
    h(\param) \eqdef \int \pi_{\param}(\rmd x) \ \H{\param}{x} \eqsp.
  \end{equation}
The following assumption is classical in stochastic approximation theory (see for
  example \cite[Part II, Section 1.6]{benveniste:metivier:priouret:1990}, or
  \cite[Section 3.3]{borkar:2008}, \cite{laruelle:pages:2012}).
\begin{assumption}
\label{hyp:LyapunovFunction} There exists a continuously differentiable function $w : \Theta \to \coint{0, \infty}$ such that
\begin{enumerate}[label=(\alph*)]
\item \label{hyp:LyapunovFunction:compacity} For any $M  > 0$, the level
  set $\{ \param \in \thetaset, w(\param) \leq M \}$ is a compact set of
  $\thetaset$.
\item \label{hyp:LyapunovFunction:calL} The set $\calL$ of stationary points, defined by
  \begin{equation}
    \label{eq:definition:setL}
    \calL \eqdef \left \{ \param \in \Theta , \pscal{\nabla w(\param)}{h(\param)}= 0 \right \}  \eqsp,
  \end{equation}
is compact.
\item \label{hyp:LyapunovFunction:scal} For any $\param \in \Theta \setminus \calL$, $\pscal{\nabla w(\param)}{h(\param)} < 0$.
\end{enumerate}
\end{assumption}
Note that under \Cref{hyp:noyau}, \Cref{hyp:Wfluctuation} and
\Cref{hyp:H:deltaH}, $h$ is H\"older-continuous on $\thetaset$
(see~\Cref{lemme:ecart:h} below). We finally provide conditions on the
stepsize sequence $\bs \pas = \suite{\pas }$.
\begin{assumption} \label{hyp:sommes:pas:poly}
  $\bs \pas = \{\gamma_0/(n+1)^{\beta}, n \in \Nset \}$ with $\gamma_0>0$ and
\[
\beta \in \ocint{ \frac{1}{p} \vee \frac{1+(\alpha \wedge \pa)/p }{1+
    \left(\alpha \wedge \pa \right) } ;1} \eqsp,
\]
where $\pa$ and $\alpha$ are respectively defined in \Cref{hyp:Wfluctuation}
and \Cref{hyp:H:deltaH}.
\end{assumption}
We denote by $\overline{\PP}_{x,\param,i}$ (resp.
$\overline{\PE}_{x,\param,i}$) the canonical probability (resp. the canonical
expectation) associated to the process $\{(X_n, \param_n, I_n), n \in \Nset \}$
defined by \Cref{algo:SA:stabilise} when $(X_0, \param_0, I_0) = (x,\param,i)$.
The main results of this contribution is summarized in the following  theorem which shows that 
\begin{enumerate}[label=(\roman*)]
\item the number of updates of the active set is finite almost surely;
\item  the process converges to the set of stationary points.
\end{enumerate}
\begin{theorem}
\label{theorem:general}
Let $\{\calK_n, n \in \Nset \}$ be a compact sequence satisfying
(\ref{eq:ConditionCompact}) and $(x_\star, \param_\star) \in \X \times
\calK_0$. Assume \Cref{hyp:PhiandH} to \Cref{hyp:sommes:pas:poly}.  The
sequence $\{(X_n, \param_n), n \in \Nset \}$ given by \Cref{algo:SA:stabilise}
started from $(x_\star, \param_\star)$ is stable:
\begin{equation}\label{theorem:general:stabilite}
   \overline{\PP}_{x_\star,\param_\star,0} \left( \bigcup_{i \geq 0}
      \bigcap_{k \geq 0} \{ \param_k \in \ \calK_i \} \right) = 1 \eqsp.
\end{equation}
If in addition, one of the following assumptions holds
\begin{enumerate}[label=(\roman*)]
\item \label{theorem:general:hyp1surw} $w(\calL)$ has an empty interior,
\item \label{theorem:general:hyp2surw} $\nabla w$ is locally Lipschitz on
  $\Theta$, \Cref{hyp:noyau}-\ref{hyp:noyau:control:W:traj} is satisfied with
  $p \geq 2$, and \Cref{hyp:sommes:pas:poly} is strengthened with the condition
  $\beta > 1/2$,
\end{enumerate}
then the sequence $\{\theta_k, k \geq 0 \}$ converges to a connected component
of $\calL$:
\begin{equation} \label{theorem:general:convergence}
   \overline{\PP}_{x_\star,\param_\star,0} \left( \lim_{k \to \infty} \mathrm{d}(\param_k, \calL)=0,  \lim_k \norm{\param_{k+1} - \param_k} =0  \right) =1 \eqsp,
 \end{equation}
 where $\mathrm{d}(x,A)$ denotes the distance from $x$ to the set $A$.
\end{theorem}
\begin{proof}
  The proof is postponed to \Cref{section:proof}.
\end{proof}

\section{Examples}
\label{section:examples}
\input{examples_HAL}

\section{Proofs}
\label{section:proof}
\input{proofregrets}

\section{Proofs of \Cref{section:examples}}
\label{section:proofs-examples}
\input{proofs-examples}

\section*{Acknowledgments}

M.~Vihola was supported by the Academy of Finland (grants 250575 and
274740).

\bibliographystyle{plain}
\bibliography{bibliographie}

\end{document}

%% file: examples_HAL.tex
For any $x \in \Rset^d$ and any $r>0$,  we define $\calB(x,r) \eqdef \{y \in
\Rset^d, \norm{y-x} \leq r\}$.

\subsection{Quantile estimation}
\label{sec:quantiles}
Let $P$ be a Markov kernel on $\X \subseteq \Rset^d$ having a stationary
distribution $\pi$. Let $\phi: \X \to \Rset$ be a measurable function. We want
to compute the quantile $\quant \in \ooint{0,1}$ under $\pi$ of the random
variable $\phi(X)$.  Quantile estimation has been considered
in~\cite[Chapter 1]{duflo:1997}; more refined algorithms can also be found
in~\cite{bardou:frikha:pages:2009,egloff:leippold:2010}. We consider the
stochastic approximation procedure $\param_{n+1} =\param_n + \gamma_{n+1}
\H{\param_n}{X_{n+1}}$ where
\begin{align} \label{H:quantiles}
  \H{\param}{x} \eqdef \quant - \indic{\phi(x) \leq {\param}} \eqsp,
\end{align}
and $\{X_k, k \geq 0\}$ is a Markov chain with Markov kernel $P$. In this
example, the Markov kernel is kept fixed \ie\ $P_\param = P$ and $\pi_\param =
\pi$ for all $\param \in \Rset$
\begin{proposition} \label{ex:quantile:d1}
  Assume that the push-forward measure of $\pi$ by $\phi$ has a density w.r.t.
  the Lebesgue measure on $\Rset$, bounded on $\Rset$ and $\int \norm{\phi(y)}
  \pi(y) \rmd y < \infty$. Assume also that \Cref{hyp:noyau} is satisfied with
  $P_\param = P$ and $\pi_\param = \pi$ for any $\param$.  Then
  \Cref{hyp:PhiandH}, \Cref{hyp:H:deltaH} and \Cref{hyp:LyapunovFunction} are
  satisfied with $\alpha =1$ and $w$, $\calL$ given by
\begin{align}
\label{eq:w:quantiles}
w({\param}) &= \frac{1}{2} \int \norm{ {\param} - \phi(y) } \pi(y) \, \rmd y +
\left( \frac{1}{2} - \quant \right) {\param} \eqsp,  \\
\calL &= \left\{ \param \in \Rset: \PP(\phi(y) \leq \param) = \quant \right\}
\eqsp.
\label{L:quantiles}
\end{align}
Furthermore, $w(\calL)$ has an empty interior.
\end{proposition}
\begin{proof}
  The proof is postponed to \Cref{sec:proof:quantile:d1}.
\end{proof}

Therefore, by \Cref{theorem:general}, \Cref{algo:SA:stabilise} applied with a
sequence $\{\gamma_n, n \geq 0\}$ satisfying \Cref{hyp:sommes:pas:poly},
provides a sequence $\{\theta_n, n \geq 0\}$ converging almost-surely to the
quantile of order $q$ of $\phi(X)$ when $X \sim \pi$.

\subsection{Stochastic Approximation  Cross-Entropy (SACE) algorithm}
\label{sec:quantile:ACE}
Let $q \in \ooint{0,1}$, $\X \subseteq \Rset^d$ and $p$ be a density on $\X$
w.r.t. the Lebesgue measure.  The goal is to find the $q$-th quantile $\theta$
of $\phi(X)$, \ie\ $\theta$ such that $\int \indic{\phi(x) \geq \theta} p(x)
\rmd x = 1-q$.  We are particularly interested in extreme quantiles, \ie\ $q
\approx 1$ for which plain Monte Carlo methods are not efficient. We consider
an approach combining MCMC and the cross-entropy method (see \eg~\cite[Chapter
13]{kroese:taimre:botev:2011}).  Let $\calP =\{ g_\nu, \nu \in \mathcal{V}
\subseteq \Rset^v\}$ be a parametric family of distributions \wrt\ the Lebesgue
measure on $\X \subseteq \Rset^d$.  The importance sampling estimator amounts
to compute, for a given value of $\theta$,
\begin{equation}
\label{eq:IS-estimate}
n^{-1} \sum_{i=1}^n \indic{\phi(Z_i) \geq \theta} w_\nu(Z_i)\eqsp,
\end{equation}
where $\sequence{Z}[i][\Nset]$ is an \iid\ sequence distributed under the
instrumental distribution $g_\nu$ and $w_\nu(z)= p(z)/ g_\nu(z)$ is the
\emph{importance weight function}.  The choice of the parameter $\nu$ is of
course critical to reduce the variance of the estimator.  The optimal
importance sampling distribution, also called the zero-variance importance
distribution, is proportional to $\indic{\phi(z) \geq \theta} p(z)$. Note that
the optimal sampling distribution is known up to a normalizing constant, which
is the tail probability of interest.

The cross-entropy method amounts to choose the parameter $\nu$ by minimizing
the Kullback-Leibler divergence of $g_\nu$ from the optimal importance
distribution, or equivalently choose $\nu = \hat{\nu}$ with
\begin{equation}
\label{eq:CE-criterion}
\hat{\nu}= \argmax_{\nu \in \mathcal{V}}   \int \log g_\nu(y) \indic{\phi(y) \geq \theta} p(y) \rmd y.
\end{equation}
This integral is not directly available but can be approximated by Markov Chain Monte Carlo,
\begin{equation}
\label{eq:CE-criterion-empirical}
\hat{\nu} = \argmax_{\nu \in \mathcal{V}} \frac{1}{m} \sum_{i=1}^m \log g_\nu(Y_i)
\end{equation}
where $\sequence{Y}[i][\Nset]$ is a Markov chain with transition kernel
$Q_\theta$ where $Q_\theta$ has stationary density $p_\theta(z) \propto
\un_{\{\phi(z) \geq \theta \}} p(z)$.

In the sequel, it is assumed that $\calP$ is a canonical exponential family,
\ie\ there exist measurable functions $S: \X \to \Rset^v$, $A:\Rset^d \to
\Rset^+$, $B: \Rset^v \to \Rset$ such that, for all $\nu \in \mathcal{V}$ and
for all $x \in \X$,
\[
g_\nu(x) = A(x) \, \exp\big(B(\nu) + \pscal{\nu}{S(x)} \big) \eqsp.
\]
In such a case, solving the optimization problem
\eqref{eq:CE-criterion-empirical} amounts to estimate the sufficient statistics
$\bar{S}_m= m^{-1} \sum_{i=1}^m S(Y_i)$ and then to compute the maximum of $\nu
\mapsto B(\nu) + \pscal{\nu}{\bar{S}_m}$.  We assume that for any $s \in
\Rset^v$, the function $\nu \mapsto B(\nu) +\pscal{\nu}{s}$ admits a unique
maximum on $\mathcal{V}$ denoted by $\hat{\nu}(s)$.  When estimating the
quantile, the value of $\theta$ is not known a priori, and the above process
should be used several times for different values of $\theta$, which may be
cumbersome.

In the Stochastic Approximation version of the Cross Entropy (SACE algorithm),
we replace the Monte Carlo approximations \eqref{eq:IS-estimate} and
\eqref{eq:CE-criterion-empirical} by stochastic approximations. Given a
sequence of step-sizes $\{\gamma_n, n \geq 0 \}$ and a family of MCMC kernels
$\{Q_\theta, \theta \in \Rset \}$ such that $Q_\theta$ admits $p(x)
\indic{\phi(x) \geq \theta}$ as unique invariant distribution, the SACE
algorithm proceeds as follows
\begin{algorithm}[ht!]
  \SetKwInOut{Input}{Input} \SetKwInOut{Output}{Output} \Input{Initial values:
    $x \in \X$, $\nu_0 \in \mathcal{V}$, $\sigma_0 \in S(\X)$ } \Output{The
    sequence $\sequence{\theta}[n][\Nset]$ } $n =0$ \tcc*{$n$: number of
    iterations} $\nu_0 =\hat{\nu}(\sigma_0)$ \; $X_0 = x$ \;
  \Repeat{convergence of $\sequence{\theta}[n][\Nset]$}{Conditionally to the
    past, draw independently $Y_{n+1} \sim Q_{\theta_n}(Y_n, \cdot)$ and
    $Z_{n+1} \sim g_{\nu_n}$ \; $\theta_{n+1} = \theta_n + \gamma_{n+1} \left(
      q - \indic{\phi(Z_{n+1}) < \theta_n} p(Z_{n+1})/g_{\nu_n}(Z_{n+1})
    \right)$ \; $\sigma_{n+1} = (1-\gamma_{n+1}) \sigma_n + \gamma_{n+1}
    S(Y_{n+1})$ \; $\nu_{n+1} = \hat{\nu}(\sigma_{n+1}) $\; $n \leftarrow n+1$}
\caption{SACE algorithm \label{algo:ACE}}
\end{algorithm}

 This algorithm can be casted into the stochastic approximation
  form $\vartheta_{n+1} = \vartheta_n + \gamma_{n+1} \H{\vartheta_n}{X_{n+1}}$, by
  setting
\[
\vartheta_n =
\begin{bmatrix}
  \theta_n \\
  \sigma_n
\end{bmatrix} \qquad X_n =
\begin{bmatrix}
  Y_n \\
  Z_n
\end{bmatrix}  \qquad  \H{(\theta, \sigma)}{y,z} =
\begin{bmatrix}
  q - \indic{\phi(z) < \theta} \ p(z)/g_{\hat{\nu}(\sigma)}(z)
  \\
  S(y) - \sigma
\end{bmatrix}
\eqsp.
\]
It is easily seen from~\Cref{algo:ACE} that $\sequence{X}[n][\Nset]$ is a
controlled Markov chain: the conditional distribution of $X_{n+1}$ given the
past is $P_{\vartheta_n}(X_n, \cdot)$ where
\[
P_{(\theta, \sigma)}((y,z), \rmd (y',z')) = Q_{\theta}(y, \rmd y') \, g_{\hat{\nu}(\sigma)}(z') \rmd z' \eqsp;
\]
this kernel possesses a unique invariant distribution with density
\[
\pi_{(\theta, \sigma)}(y,z) = p_\theta(y) \, g_{\hat{\nu}(\sigma)}(z) \, \eqsp.
\]
Therefore, the mean field function $h$ is given by (up to a transpose)
\[
(\theta, \sigma) \mapsto \left( \int \un_{ \{\phi(z) \geq \theta\}} p(z) \rmd z
  - q, \ \int S(y) p_\theta(y) \rmd y - \sigma \right) \eqsp.
\]
We establish that SACE satisfies \Cref{hyp:H:deltaH} in the case $\calP$ and
$\pi$ satisfy the following assumptions
\begin{assumptionE}\label{hyp:ACE}
  \begin{enumerate}
  \item  There exists $\alpha \in \ocint{0,1}$ and for any compact set $\calK$ of
  $S(\X)$, there exists a constant $C$ such that $\left| B(\hat \nu(\sigma)) -
    B(\hat \nu (\sigma')) \right| + \norm{\hat \nu(\sigma) - \hat \nu(\sigma')} \leq C
  \norm{\sigma - \sigma'}^{\alpha}$ for any $\sigma, \sigma' \in \calK$.
\item  The push-forward distribution of $p$ by $\phi$ possesses a bounded density
  w.r.t. the Lebesgue measure on $\Rset$.  In addition, there exists $\delta
  >0$ such that
\[
\int \left(1+ \norm{S(x)} \right) \exp\left(1+ \delta \norm{S(x)} \right) \,
p(x) \rmd x < \infty \eqsp.
\]
  \end{enumerate}
\end{assumptionE}
\begin{proposition}
\label{prop:convergence-ACE}
Assume \Cref{hyp:ACE}. Then \Cref{hyp:H:deltaH} holds with $\alpha$ given by
\Cref{hyp:ACE}.
\end{proposition}
\begin{proof}
  The proof is postponed to~\Cref{sec:proof:prop:convergence-ACE}.
\end{proof}

Consider a bridge network: the network is composed with nodes and $d$ edges $\{e_1, \cdots, e_d \}$ with length $\{U_\ell, \ell \leq d \}$.  Fix two nodes $N_1,N_2$ in the graph; we are interested in the length of the shortest path from $N_1$ to $N_2$ defined by
\[
\phi(U_1, \cdots, U_d) = \mathrm{min}_{\calC} \sum_{\ell \, \text{s.t.} \,
  e_\ell \in \calC} \bu_\ell \, U_\ell
\]
where $\calC$ denotes a path from $N_1$ to $N_2$ ($\calC$ is a set of edges)
and $\bu_\ell >0$ (see \Cref{fig:GraphBridgeNetwork}).
\begin{figure}[htbp]
  \centering
  \includegraphics[width=5cm]{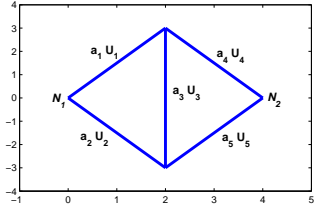}
\caption{A bridge network, in the case $d=5$}
  \label{fig:GraphBridgeNetwork}
\end{figure}
  It is assumed that the lengths $\{U_\ell , \ell \leq d \}$ are i.i.d. and uniformly distributed on $\ccint{0, 1}$. We are interested in
computing a threshold $\theta_\star$ such that the probability $\phi(U_1,
\cdots, U_d)$ exceeds $\theta_\star$ is $1-q$ in the case $q$ is close to one.
In this example,
\[
\X = \ccint{0, 1}^d \eqsp, \qquad p(u_1, \cdots, u_d) = \prod_{\ell=1}^d
\indic{\ccint{0,1}}(u_\ell) \eqsp.
\]
The importance sampling distribution $g_\nu$ is a product of Beta$(\nu_\ell,1)$
distributions:
\[
\mathcal{V} = \left(\Rset^+ \setminus \{0 \} \right)^d \eqsp, \qquad g_\nu(u_1,
\cdots, u_d) = \prod_{\ell=1}^d \nu_\ell \left( u_\ell \right)^{\nu_\ell-1} \indic{\ccint{0,1}}(u_\ell) \eqsp.
\]
$g_\nu$ is from a canonical exponential family with $S(u_1, \cdots, u_d) =
\left( \ln u_\ell \right)_{1 \leq \ell \leq d}$ and $B(\nu_1, \cdots, \nu_d) =
\sum_{\ell=1}^d \ln \left( \nu_\ell \right)$.  Furthermore, for any $1 \leq
\ell \leq d$, $\left(\hat \nu(s_1, \cdots, s_d) \right)_\ell = \left( - s_\ell
\right)^{-1}$.

In this example, $\Theta = \Rset \times \prod_{\ell=1}^d \ooint{-\infty, 0}$.
The assumptions \Cref{hyp:ACE} are easily verified with $\alpha=1$ (details are
omitted). For the MCMC samplers $\{Q_\theta, \theta \in \Rset \}$, we use a
Gibbs sampler: note that for any $\ell \in \{1, \cdots, d \}$ and any $\{u_j, j
\neq \ell \}$, $u_\ell \mapsto \phi(u_1, \cdots, u_d)$ is increasing.  Hence,
the conditional distribution of the $\ell$-th variable conditionally to the
others when the joint distribution is proportional to $p(x) \indic{\phi(x) \geq
  \theta}$ is a uniform distribution on $\ccint{(\theta - \phi(u^\ast))_+,1}$
where $u^\ast =(u_1, \cdots, u_{\ell-1}, 0, u_{\ell+1}, \cdots, u_d)$.

We illustrate the convergence of SACE for the bridge network displayed on
\Cref{fig:GraphBridgeNetwork} in the case
$[\bu_1, \cdots, \bu_d] = [1, 2, 3, 1,2]$. On
\Cref{fig:cvgSACE}[left], we show a path of
$\{\theta_n, n\geq 0 \}$ for different runs of the algorithm corresponding to
the different quantiles $q$. On \Cref{fig:cvgSACE}[right],
we show the path of the $d$ components of
$\hat{\nu}(\sigma_n)$ in the case $q = 0.001$.  Not surprisingly, the largest
values of the parameters $\hat{\nu}(\sigma_n)_\ell$ at convergence are reached
with $\ell = 1, 4$; they correspond to the shortest range of path length ($a_1
= a_4 =1$).
\begin{figure}[h]
  \centering
\begin{tabular}{cc}
  \includegraphics[width=0.5\textwidth]{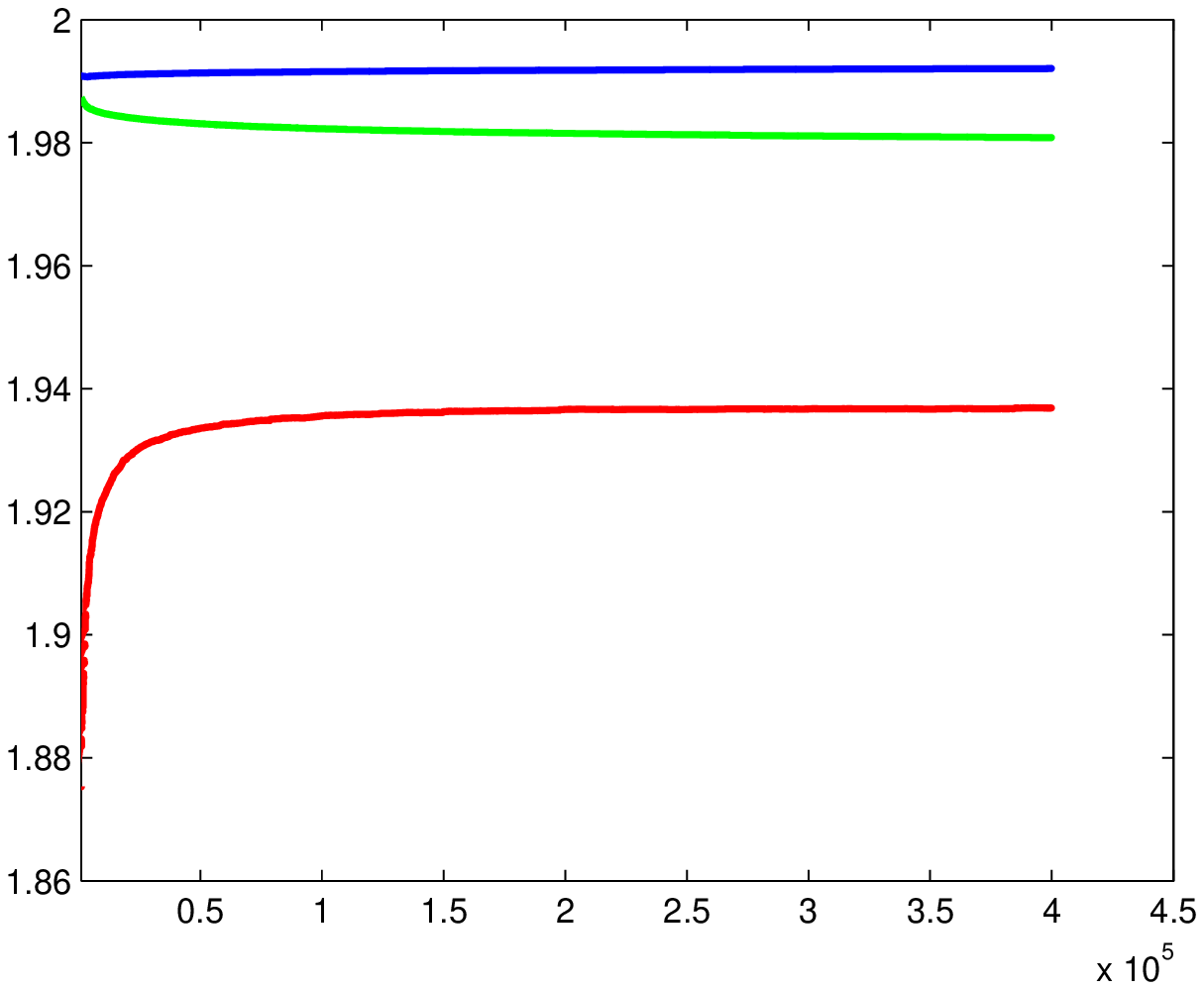}  &  \includegraphics[width=0.5\textwidth]{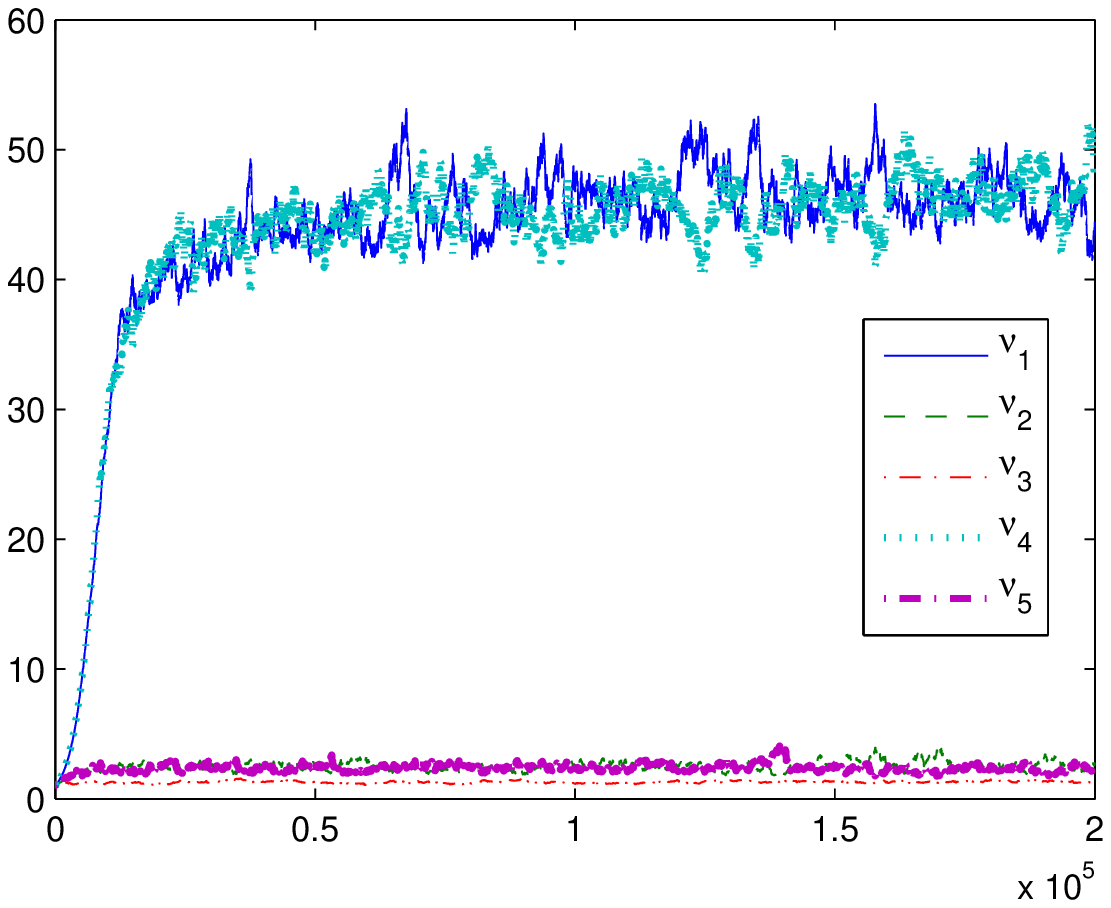}
\end{tabular}
  \caption{[left] Paths of $\theta_n$ for different values of $q$ ($q=10^{-3}, 10^{-4}$ and $10^{-5}$) - the first $1000$ iterations are discarded. [Right] Path of the $d$ components of $\hat{\nu}(\sigma_n)$, when $q = 0.001$}
\label{fig:cvgSACE}
\end{figure}

\subsection{Median in multi-dimensional spaces}
\label{sec:quantile:median}
In a multivariate setting different extensions of the median have been
proposed in the literature (see for instance \cite{cardot:cenac:zitt:2012} and
\cite{cardot:cenac:zitt:2013} and the references therein).
We focus here on the spatial median, also named geometric median which is probably the most frequently used. The median $\theta$ of a random vector $X$ taking values in $\rset^d$ with $d \geq 2$
is
\[
\param =  \argmin_{m \in \rset^d} \PE[ \norm{X-m} ]
\]
The median $\param$ is uniquely defined unless the support of the distribution
of $X$ is concentrated on a one dimensional subspace of $\rset^d$. Note also that it is translation invariant. Following \cite{cardot:cenac:zitt:2012}, we consider the stochastic approximation procedure with
\begin{align} \label{H:mediane}
 \H{\param}{x} \eqdef \frac{x-\param}{\norm{x-\param} } \indic{x \neq \param} \eqsp,
\end{align}
and the following assumptions
\begin{assumptionE}\label{hyp:quantile:median}
  the condition \Cref{hyp:noyau} is satisfied with $\pi_\param = \pi$ for any $\param$. The density $\pi$ is bounded on $\Rset^d$ and $\int
  \norm{x} \pi(x) \rmd x < \infty$.
\end{assumptionE}
\begin{proposition}
\label{prop:example-quantile}
Assume \Cref{hyp:quantile:median}.  Let $\param_\star$ be the unique solution
of $\int \H{\param}{x} \pi(x) \rmd x = 0$.  Then \Cref{hyp:PhiandH},
\Cref{hyp:H:deltaH} and \Cref{hyp:LyapunovFunction} are satisfied with $w$ and
$\calL$ given by
  \[
  w(\param) = \int \norm{x-\param} \pi(x) \rmd x \eqsp, \qquad \calL =
  \{\param_\star \} \eqsp,
  \]   and with  $\alpha = d/(1+d)$.
\end{proposition}
\begin{proof}
The proof is postponed to \Cref{sec:proof:quantile:median}.
\end{proof}

Here again, $w(\calL)$ has an empty interior; and \autoref{theorem:general}
provides sufficient conditions on the kernels $\{P_\theta, \theta \in \Theta
\}$ and on the sequence $\{\gamma_n, n \geq 0\}$ implying that $\{\theta_n, n
\geq 0 \}$ converges almost-surely to $\param_\star$.

\subsection{Vector quantization}
\label{sec:vectquant}
\label{sec:qv:cv}
Vector quantization consists of approximating a random vector $X$ in $\Rset^d$
by a random vector taking at most $N$ values in $\Rset^d$. In this section, we
assume that
\begin{assumptionE}
\label{hyp:qv}
the distribution of $X$ is absolutely continuous with respect to the Lebesgue
measure on $\Rset^d$ with density $\pi$ having a bounded support: $\pi(x) = 0$
for any $\norm{x} > \Delta$ for some $\Delta >0$.
\end{assumptionE}
Vector quantization plays a crucial role in source coding
\cite{juang:etal:1981}, numerical integration \cite{pages:1998,pages:etal:2004}
and nonlinear filtering~\cite{pages:pham:2005,pham:runggaldier:sellami:2005}.
Several stochastic approximations procedure have been proposed to
approximate the optimal quantizer; see  \cite{kohonen:1982,benaim:fort:pages:1998}.
For $\btheta = (\theta^{(1)}, \theta^{(2)}, \dots, \theta^{(N)}) \in
(\Rset^d)^N$, and for any $1 \leq i \leq N$, define the Voronoi cells
associated to the dictionary $\theta$ by
\begin{align*}
  {C}^{(i)}(\btheta) \eqdef \left\{ u \in \Rset^d, \norm{u-\theta^{(i)}} = \min
    \limits_{1 \leq j \leq N} \norm{u - \theta^{(j)}} \right\} \eqsp.
\end{align*}
These cells allow to approximate a random vector $X$ by $\theta^{(i)}$ in the
cell $C^{(i)}(\btheta)$. Denote by $\widetilde w$ the mean squared quantization
error (or distortion) given by
\begin{align*}
  \widetilde w(\btheta) \eqdef \sum_{i=1}^N \PE_\pi \left[ \norm{X -
      \theta^{(i)} }^2 \indic{C^{(i)}(\btheta)}(X) \right] \eqsp.
\end{align*}
The Kohonen algorithm (with 0 neighbors) is a stochastic approximation
algorithm with field $\widetilde H : (\Rset^d)^N \times \Rset^d \to
(\Rset^d)^N$ given by
\begin{align*}
\tH{\btheta}{u} \eqdef - 2 \left((\theta^{(i)}-u)
  \indic{C^{(i)}(\btheta)}(u) \right)_{1 \leq i \leq N} \eqsp.
\end{align*}
The convergence of the Kohonen algorithm has been established
in dimension $d=1$ for \iid\ observations $\{X_n, n \in
\Nset\}$~\cite{pages:1998}. The case $d \geq 2$ is still an open question: one
of the main difficulty arises from the non-coercivity of the distortion
$\widetilde w$ and the non-smoothness property of the field $\tH{\theta}{x}$.
The goal here is to go a step further in the study of the multidimensional
case, including a generalization to the case the observations $\{X_k, k\geq 1
\}$ are Markovian (see
e.g.~\cite{pages:etal:2004,pham:runggaldier:sellami:2005}):
\begin{assumptionE}
 \label{hyp:qv:kern}
 $\{P_{\btheta}, \btheta \in \thetaset\}$ is a family of kernels satisfying
 \Cref{hyp:noyau} for some $p \geq 2$, \Cref{hyp:Wfluctuation} and such that
 $\pi_{\btheta} = \pi$ for any $\btheta \in \thetaset$.
\end{assumptionE}
We define
\[
\Theta = \left\{ \btheta = (\theta^{(1)}, \dots \theta^{(N)}) \in (\Rset^d)^N
  \cap (\mathcal{B}(0,\Delta))^N, \theta^{(i)} \neq \theta^{(j)} \ \text{for all $i
  \neq j$} \right\} \eqsp;
\]
and run \Cref{algo:SA:stabilise} with $\H{\btheta}{} \in \rset^{dN}$ defined by
\begin{equation}\label{eq:qv:defH}
\H{\btheta}{x} = \tH{\btheta}{x} - \lambda \left( \sum_{j \neq i} \frac{\theta^{(i)} - \theta^{(j)}}{\norm{\theta^{(i)} - \theta^{(j)}}^4}\right)_{1 \leq i
\leq N}
\end{equation}
for some $\lambda >0$ with the sequence of compact sets
$\calK_q = \{\btheta \in \Theta: \min_{i \neq j} \norm{\theta^{(i)} -
  \theta^{(j)}} \geq 1/q \}$.
Set
\begin{align}\label{eq:qv:defw}
  w(\btheta) = \widetilde{w}(\btheta) + \frac{\lambda}{4} \sum_{i \neq j} \norm{\theta^{(i)} - \theta^{(j)}}^{-2}.
\end{align}
Under \Cref{hyp:qv}, from~\cite[Proposition~9 and Lemma 29]{pages:1998}, $w$ is
continuously differentiable on $\Theta$ and $\nabla w(\btheta) = - \int
\H{\btheta}{x} \pi(x) \, \rmd x$.  Furthermore, for any compact set $\calK_q$,
there exists a constant $C$ such that for any $\btheta \in \calK_q$, $\bar
\btheta \in \calK_q$, $\norm{\nabla w(\theta) - \nabla w(\bar \btheta)} \leq C \norm{\btheta - \bar
  \btheta}$. Hence, $\nabla w$ is locally Lipschitz on $\Theta$.
\begin{lemma}
\label{prop:qv:cv}
Assume \Cref{hyp:qv} and \Cref{hyp:qv:kern}. Then \Cref{hyp:PhiandH},
\Cref{hyp:H:deltaH} and \Cref{hyp:LyapunovFunction} are satisfied with $w$
defined by \eqref{eq:qv:defw} and $\calL \eqdef \{ \theta \in \thetaset: \nabla
w(\theta)=0\}$.
\end{lemma}
\begin{proof}
  The proof is postponed in  \Cref{proof:sec:vectquant}.
\end{proof}

\Cref{theorem:general} shows that under \Cref{hyp:qv}, \Cref{hyp:qv:kern}, the
penalized $0$-neighbors Kohonen algorithm converges a.s. to a connected
component of $\calL$.

%% file: proofregrets.tex
In \Cref{sec:SuccCond:SA}, we start with preliminary results on the stability
and the convergence of stochastic approximation schemes. We provide a new set
of sufficient conditions for the convergence of stable SA algorithms
(see~\Cref{lemma:rappel:AMP:item2}). In \Cref{sec:poisson}, some properties of
the Poisson equations associated to the transition kernels $\{P_\param, \param
\in \Theta\}$ are discussed. In \Cref{sec:proof:theorem:general}, we first
state a control of the perturbations $\H{\theta_n}{X_{n+1}} -h(\theta_n)$ (see
\Cref{prop:controle:Sln}) which is the key ingredient for the proof of
\Cref{theorem:general}; we then conclude \Cref{sec:proof:theorem:general} by
giving the proof of \Cref{theorem:general}.  The proof of
\Cref{prop:controle:Sln} is given in \Cref{sec:proof:prop:controle:Skl}.

\subsection{Stability and Convergence of Stochastic Approximation algorithms}
\label{sec:SuccCond:SA}
Let  $\{\vartheta_n, n \geq 0 \}$ be defined, for all $\vartheta_0 \in \thetaset$ and $n \geq 1$ by:
\begin{equation}
\label{eq:definition:vartheta}
\vartheta_n = \vartheta_{n-1} + \rho_n h(\vartheta_{n-1}) + \rho_n \xi_n
 \eqsp,
\end{equation}
where $\{\rho_n, n \in \Nset \}$ is a sequence of positive numbers and $\{\xi_n, n\in \Nset \}$ is a sequence of $\Rset^d$-vectors.  For any $L \geq
0$, the sequence $\{\tilde \vartheta_{L,k}, k \geq 0 \}$ is defined by $\tilde \vartheta_{L,0} = \vartheta_L$, and for $k \geq 1$,
\begin{equation}
\label{eq:definition:tilde-vartheta}
\tilde \vartheta_{L,k} =\tilde \vartheta_{L,k-1} + \rho_{L+k} h(\vartheta_{L+k-1}) \eqsp.
\end{equation}
\begin{lemma}
  \label{lemma:rappel:AMP:item1}
  Assume \Cref{hyp:LyapunovFunction}.  Let $M_0$ be such that $\{\vartheta_0 \}
  \cup \calL \subset \{\vartheta: w(\vartheta) \leq M_0 \}$.  There exist
  $\delta_\star >0$ and $\lambda_\star >0$ such that for any non-increasing
  sequence $\{\rho_n, n \in \Nset \}$ of positive numbers and any
  $\Rset^d$-valued sequence $\{\xi_n, n\in \Nset \}$
\[
\left( \rho_0 \leq \lambda_\star \ \text{and} \ \sup_{k \geq 1} \norm{
    \sum_{j=1}^k \rho_j \xi_j }\leq \delta_\star \right) \Longrightarrow \left(
  \sup_{k \geq 1} w(\vartheta_k) \leq M_0+1 \right) \eqsp.
\]
\end{lemma}
\begin{proof}
  See \cite[Theorem 2.2]{andrieu:moulines:priouret:2005}.
\end{proof}

\begin{lemma}
\label{lem:determ:GalUnifCont}
Let $g : \Theta \to \Rset$ be a continuous function. For any compact set
$\calK \subset \Theta$ and $\delta>0$, there exists $\eta >0$ such that for all
$\vartheta \in \calK$ and $\vartheta' \in \Theta$ satisfying $|\vartheta-\vartheta'| \leq \eta$, $|g(\vartheta) - g(\vartheta')| \leq \delta$.
\end{lemma}

\begin{lemma}
\label{lem:determ:ContractionV}
Assume \Cref{hyp:LyapunovFunction}-\ref{hyp:LyapunovFunction:scal} and $h$ is continuous.  For any compact set
$\calK$ of $\Theta$ such that $\calK \cap \calL = \emptyset$ and any $\delta
\in \ooint{0, \inf_\calK |\pscal{\nabla w}{h}|}$, there exist $\lambda >0,\beta
>0$ such that for all $ \vartheta \in \calK$, $\rho \leq \lambda$ and $|\xi|
\leq \beta$, $w\left(\vartheta + \rho h(\vartheta) + \rho \xi \right) \leq
w(\vartheta) - \rho \delta$.
\end{lemma}
\begin{proof}
  See \cite[Lemma 2.1(i)]{andrieu:moulines:priouret:2005}.
\end{proof}

\begin{lemma}
  \label{lem:CvdDeltaUn}
  Assume \Cref{hyp:LyapunovFunction}, $h$ is continuous, $\lim_n \rho_n = 0$,
  $\lim_k \sum_{j=1}^k \rho_j \xi_j$ exists and there exists $M > 0$ such that
  for any $n \geq 0$, $\vartheta_n \in \calK \eqdef \{\theta \in \Theta:
  w(\theta) \leq M \}$. Then
  \begin{enumerate}[label=(\roman*)]
  \item \label{Lem:Utilde:item:i} for any $L \geq 0$ and $k \geq 0$, $\vartheta_{L+k+1}
    - \tilde \vartheta_{L,k+1} = \sum_{j=L}^{L+k}\rho_{j+1} \xi_{j+1}$.
\item \label{Lem:Utilde:item:ii}  $\limsup_{n \to \infty} | \vartheta_{n+1} - \vartheta_n | = 0$.
\item \label{Lem:Utilde:item:iii} for any $\widetilde M > M$, there exists
  $\tilde L$ such that for any $L \geq \tilde L$, $\{\tilde \vartheta_{L,k}, k
  \geq 0\} \subset \widetilde \calK \eqdef \{\theta \in \Theta: w(\theta) \leq
  \widetilde M \}$.
\item \label{Lem:Utilde:item:iv}   $\lim_{L \to \infty} \sup_{l \geq 0} | w(\tilde \vartheta_{L,l}) - w(\vartheta_{L+l})|
  =0$.
  \end{enumerate}
\end{lemma}
\begin{proof}
  \begin{enumerate}[label=(\roman*),wide=0pt,labelindent=\parindent]
  \item \eqref{eq:definition:tilde-vartheta} implies that $\tilde \vartheta_{L,k+1} - \vartheta_{L+k+1} = \tilde
    \vartheta_{L,k} - \vartheta_{L+k} - \rho_{L+k+1} \xi_{L+k+1}$ from which
    the proof follows.
  \item Under \Cref{hyp:LyapunovFunction}-\ref{hyp:LyapunovFunction:compacity}
    and the continuity of $h$ on $\Theta$, $\sup_{\calK} |h| < +\infty$.  Since $\vartheta_k \in \calK$
    for any $k$, we get $|\tilde \vartheta_{L,k+1} - \tilde \vartheta_{L,k}|
    \leq \rho_{L+k+1} \sup_\calK |h|$.  Let $\epsilon >0$.  Under the stated
    assumptions, we may choose $K_\epsilon$ such that for any $k \geq
    K_\epsilon$ and  $L \geq 0$, $|\tilde \vartheta_{L,k} - \tilde
      \vartheta_{L,k-1} | \leq \epsilon$ and $L_\epsilon$  such that for any $L \geq L_\epsilon$, $\sup_{l \geq 1} |\sum_{j=L}^{L+l}
    \rho_{j+1} \xi_{j+1} \un_{\vartheta_j \in \calK} | \leq \epsilon$.  By
    \ref{Lem:Utilde:item:i}, for any $k \geq K_\epsilon$ and $L = L_\epsilon$,
\begin{equation*}
  \left|\vartheta_{L+k+1} - \vartheta_{L+k} \right|  \leq \left| \tilde \vartheta_{L,k+1} - \tilde
    \vartheta_{L,k} \right|+ \left|\vartheta_{L+k+1}
    -\tilde \vartheta_{L,k+1} \right|+ \left|\tilde \vartheta_{L,k} - \vartheta_{L+k}\right|  \leq 3 \epsilon \eqsp.
\end{equation*}
\item $\widetilde \calK$ is compact by
  \Cref{hyp:LyapunovFunction}-\ref{hyp:LyapunovFunction:compacity}.  By
  \Cref{lem:determ:GalUnifCont}, there exists $\eta>0$ such that for all
  $\vartheta \in \calK$, $\vartheta' \in \Theta$, $|\vartheta-\vartheta'| \leq
  \eta$, $\left|w(\vartheta) -w(\vartheta') \right| \leq \widetilde{M} - M$.
  There exists $\tilde L$ such
  that for any $L \geq \tilde L$, $\sup_{l \geq 1} |\sum_{j=L}^{L+l} \rho_{j+1}
  \xi_{j+1} \un_{\vartheta_j \in \calK} | \leq \eta$.  By
  \ref{Lem:Utilde:item:i}, $\sup_{L \geq \tilde L} \sup_{k \geq 0} | \tilde
  \vartheta_{L,k} - \vartheta_{L+k}| \leq \eta$.  Since $\vartheta_j \in \calK$
  for any $j \geq 0$, this implies that for any $L \geq \tilde L$ and $k \geq
  0$, $w(\tilde \vartheta_{L,k}) \leq \widetilde M$ and $\tilde \vartheta_{L,k}
  \in \widetilde \calK$.
\item The proof is on the same lines as the proof of \ref{Lem:Utilde:item:iii}.
  \end{enumerate}
\end{proof}

 \begin{lemma}\label{lemma:RobbinsSiegmund:deterministe}
    Let $\{v_n, n\geq 0\}$ and $\{\chi_n, n\geq 0\}$ be non-negative sequences
    and $\{\eta_n, n \geq 0\}$ be a sequence  such that $\sum_n \eta_n$ exists.  If for any
    $n \geq 0$, $v_{n+1}\leq v_n - \chi_n + \eta_n$ then $\sum_n \chi_n <
    \infty$ and $\lim_n v_n$ exists.
\end{lemma}
\begin{proof}
  Set $w_n =v _n + \sum_{k \geq n} \eta_k +M$ with $M \eqdef - \inf_n \sum_{k
    \geq n} \eta_k$ so that $\inf_n w_n \geq 0$. Then
\[
0 \leq w_{n+1} \leq v_n - \chi_n + \eta_n + \sum_{k \geq n+1} \eta_k +M \leq
w_n -\chi_n \eqsp.
\]
Hence, the sequence $\{w_n, n \geq 0\}$ is non-negative and non-increasing;
therefore it converges.  Furthermore, $0 \leq \sum_{k=0}^n \chi_k \leq w_0$ so
that $\sum_n \chi_n < \infty$. Therefore, the convergence of $\{w_n, n\geq 0\}$
also implies the convergence of $\{v_n, n\geq 0\}$. This concludes the proof.
\end{proof}

\begin{proposition}
  \label{lemma:rappel:AMP:item2}
  Assume \Cref{hyp:LyapunovFunction}.  Let $\{\rho_n, n \in \Nset
  \}$ be a non-increasing sequence of positive numbers and $\{\xi_n, n\in \Nset
  \}$ be a sequence of $\Rset^d$-vectors.   Assume
  \begin{enumerate}[label=(C-\roman*)]
  \item \label{lemma:rappel:AMP:item2:hyp1} $h: \Theta \to \Rset^d$ is continuous.
  \item \label{lemma:rappel:AMP:item2:hyp2}  $\{\vartheta_k, k \in \Nset \} \subset \calK \eqdef \{ \theta \in \Theta: w(\theta) \leq M \}$.
  \item \label{lemma:rappel:AMP:item2:hyp3} $\sum_k \rho_k = +\infty$ and $\lim_k \rho_k = 0$.
  \item \label{lemma:rappel:AMP:item2:hyp4} $\lim_{k \to \infty} \sum_{j=1}^k \rho_j \xi_j$ exists.
  \item \label{lemma:rappel:AMP:item2:hyp5} one of the following conditions
    \begin{enumerate}[label=(\Alph*)]
    \item \label{lemma:rappel:AMP:item2:hyp5a} $w(\calL)$ has an empty interior
    \item \label{lemma:rappel:AMP:item2:hyp5b} $\nabla w$ is locally Lipschitz
      on $\Theta$, and the series $\sum \rho_j \pscal{\nabla w(\vartheta_j)}{
        \xi_j}$ and  $\sum_j \rho_{j}^2 \norm{\xi_{j}}^2$ converge.
    \end{enumerate}
  \end{enumerate}
  Then $\{\vartheta_n, n \geq 0 \}$ converges to a connected component of
  $\calL$.
\end{proposition}

\begin{proof}
Assume first that $\lim_{n \to \infty} w(\vartheta_n)$ exists.

  $\blacktriangleright$ {\tt Step 1.}  For $\alpha>0$, let $\calL_\alpha \eqdef
  \{\theta \in \Theta: \mathrm{d}(\theta, \calL) < \alpha \}$ be the
  $\alpha$-neighborhood of $\calL$.  Let $\epsilon >0$. We prove that there
  exist $L_{\epsilon}$ and $\delta_1>0$ such that for any $L \geq L_{\epsilon}$,
\begin{enumerate}[label=(\alph*)]
\item \label{theo:determ:cvgU:step1c} $\sup_{k \geq 0} \left| \tilde \vartheta_{L,k} - \vartheta_{L+k}\right| \leq \epsilon$ and
$\sup_{k \geq 0} \left| w(\tilde \vartheta_{L,k}) - w(\vartheta_{L+k})\right| \leq \epsilon$.
\item \label{theo:determ:cvgU:step1b} for any $k \geq 0$, it holds: $\tilde
  \vartheta_{L,k} \notin \calL_\alpha \Longrightarrow w(\tilde \vartheta_{L,k+1}) - w(\tilde
  \vartheta_{L,k}) \leq - \rho_{L+k+1} \delta_1$.
\item \label{theo:determ:cvgU:step1a} the sequence $\{\tilde \vartheta_{L,k}, k \geq
  0\}$ is infinitely often in $\calL_\alpha$.
\end{enumerate}
Let $\widetilde M > M$ and set $\widetilde \calK \eqdef \{\theta \in \Theta:
w(\theta) \leq \widetilde M \}$.
Note that by
\Cref{hyp:LyapunovFunction}-\ref{hyp:LyapunovFunction:compacity}, $\widetilde
\calK \subset \Theta$ is  compact  and since $\calL_\alpha$ is open,
$\widetilde\calK^\alpha \eqdef \widetilde \calK \setminus \calL_\alpha$ is
compact and $ \widetilde \calK^\alpha \cap \calL =
\emptyset$.  By \Cref{lem:determ:ContractionV}, there exist $\delta_1 >0$,
$\lambda_1>0$, $\beta_1>0$ such that
\begin{equation}\label{eq:CvgU:tool3}
\left(\vartheta \in \widetilde \calK^\alpha, \rho \leq \lambda_1, |\xi| \leq \beta_1 \right)
\Longrightarrow w\left(\vartheta + \rho h(\vartheta) + \rho \xi \right) \leq w(\vartheta) - \rho
\delta_1 \eqsp.
\end{equation}
By \Cref{lem:determ:GalUnifCont}, there exists $\delta_2 \in
\ooint{0,\epsilon}$ such that
\begin{equation}\label{eq:CvgU:tool4}
\left(\vartheta \in \calK, \vartheta' \in \Theta, |\vartheta-\vartheta'| \leq \delta_2 \right)\Longrightarrow
\left( \left| h(\vartheta) - h(\vartheta')\right| \leq \beta_1, \ \  \left| w(\vartheta) - w(\vartheta')\right| \leq  \epsilon \right)\eqsp.
\end{equation}
By \Cref{lem:CvdDeltaUn}-\ref{Lem:Utilde:item:iii}, \ref{lemma:rappel:AMP:item2:hyp3} and \ref{lemma:rappel:AMP:item2:hyp4},
there exists $\widetilde L$ such that for any $L \geq \widetilde L$, $\{ \tilde
\vartheta_{L,k}, k \geq 0\} \subset \widetilde \calK$ and
\begin{equation}\label{eq:CvgU:tool2}
\sup_{k \geq 0} \rho_{L+k} \leq \lambda_1, \qquad \sup_{l \geq 1} \left|\sum_{j=L}^{L+l} \rho_{j+1} \xi_{j+1} \un_{\vartheta_j \in
    \calK} \right| \leq \delta_2 \eqsp.
\end{equation}
\begin{enumerate}[label=(\alph*),wide=0pt,labelindent=\parindent]
\item follows from \Cref{lem:CvdDeltaUn}-\ref{Lem:Utilde:item:i} and \eqref{eq:CvgU:tool2}.
\item Let $L \geq \tilde{L}$, $k \geq 0$ and $\tilde \vartheta_{L,k} \not \in \calL_\alpha$. The proof follows from \eqref{eq:CvgU:tool3} and
\[
\tilde \vartheta_{L,k+1} = \tilde \vartheta_{L,k} + \rho_{L+k+1} h\left( \tilde
  \vartheta_{L,k}\right) + \rho_{L+k+1} \left( h\left( \vartheta_{L+k}\right) - h\left(
    \tilde \vartheta_{L,k}\right) \right) \eqsp,
\]
using  \eqref{eq:CvgU:tool4} and \eqref{eq:CvgU:tool2}.
\item The proof is by contradiction. Let $L \geq \widetilde L$, $k \geq 0$ and assume that for any $j \geq 0$, $\tilde
\vartheta_{L,k+j} \in \widetilde \calK^\alpha$. By \ref{theo:determ:cvgU:step1b}, for any $j,k \geq 0$,
$ w\left(\tilde \vartheta_{L,k+j+1} \right) \leq w(\tilde \vartheta_{L,k+j}) - \rho_{L+k+j+1}
\delta_1$  which  implies under \ref{lemma:rappel:AMP:item2:hyp3} that
$\lim_{j \to \infty} w(\tilde \vartheta_{L,k+j}) = -\infty$.  Since $w$ is continuous and  nonnegative and $\calK$ compact,
 $\inf_{\widetilde \calK} w \geq 0$. This is a contradiction.
\end{enumerate}
$\blacktriangleright$ {\tt Step 2.}  Let $\alpha >0$ and $\epsilon >0$.  By
Step 1 and Lemma~\ref{lem:CvdDeltaUn}-\ref{Lem:Utilde:item:iv}, there exists
$L_{\epsilon}$ such that for any $L \geq L_{\epsilon}$ and $k \geq 0$, $\sigma_{L,k}
\eqdef \inf \{j \geq 0, \tilde \vartheta_{L, k+j} \in \calL_\alpha \}$ is
finite, $\sup_{k \geq 0} \left| \tilde \vartheta_{L,k} - \vartheta_{L+k}\right| \leq
  \epsilon$, $\sup_{k \geq 0} \left| w(\tilde \vartheta_{L,k}) -
    w(\vartheta_{L+k})\right| \leq \epsilon$ and
\begin{equation*}
\delta_1 \sum_{\ell=1}^{\sigma_{L,k}} \rho_{L+k+\ell} \leq w(\tilde
  \vartheta_{L,k}) - w(\tilde \vartheta_{L,k+\sigma_{L,k}})
 \eqsp.
\end{equation*}
Hence, for any $k \geq 0$, using $| \tilde \vartheta_{L, \ell} - \tilde \vartheta_{L,\ell-1}|
\leq \rho_{L+\ell} \sup_{\calK} |h|$,
\begin{align*}
  \mathrm{d}\left(\vartheta_{L+k}, \calL_\alpha \right) & \leq \left|
    \vartheta_{L+k} - \tilde \vartheta_{L,k+\sigma_{L,k}} \right| \leq \left|
    \vartheta_{L+k} - \tilde \vartheta_{L,k} \right| + \left|
    \tilde \vartheta_{L,k} - \tilde \vartheta_{L,k+\sigma_{L,k}} \right| \\
  & \leq \epsilon + \sum_{\ell=1}^{\sigma_{L,k}} \left| \tilde
    \vartheta_{L,k+\ell} -
    \tilde \vartheta_{L,k+\ell-1} \right| \leq \epsilon + \sup_{\calK} |h| \,
  \sum_{\ell=1}^{\sigma_{L,k}} \rho_{L+k+\ell} \\
  & \leq \epsilon + \delta_1^{-1} \sup_{\calK} |h| \, \sup_{\ell \geq 0} \left|
    w(\tilde \vartheta_{L,k}) - w(\tilde \vartheta_{L,k+\ell}) \right| \eqsp,  \\
  &\leq \epsilon + 2 \epsilon \delta_1^{-1} \sup_{\calK} |h| + \delta_1^{-1}
  \sup_{\calK} |h| \, \sup_{\ell \geq 0} \left| w(\vartheta_{L+k}) -
    w(\vartheta_{L+k+\ell}) \right| \eqsp.
\end{align*}
This proves that $\lim_n \mathrm{d}(\vartheta_n, \calL) =0$ since $\lim_n
w(\vartheta_n)$ exists. Finally, since by \Cref{lem:CvdDeltaUn}-\ref{Lem:Utilde:item:ii},
$\lim_n |\vartheta_n - \vartheta_{n-1}|
=0$, the sequence
$\{\vartheta_n, n \geq 0 \}$ converges to a connected component of $\calL$.

$\blacktriangleright$ {\tt Step 3.} We now prove that $\lim_n w(\vartheta_n)$
exists.  Under \ref{lemma:rappel:AMP:item2:hyp5}-\ref{lemma:rappel:AMP:item2:hyp5a}, the proof
follows from~\cite[Theorem 2.3]{andrieu:moulines:priouret:2005}. We prove that
this limit exists under \ref{lemma:rappel:AMP:item2:hyp5}-\ref{lemma:rappel:AMP:item2:hyp5b}.

By \Cref{hyp:LyapunovFunction}-\ref{hyp:LyapunovFunction:compacity},
\ref{lemma:rappel:AMP:item2:hyp2} and
\Cref{lem:CvdDeltaUn}-\ref{Lem:Utilde:item:ii}, there exist $N$ and a compact
set $\widetilde \calK$ of $\Theta$ such that $\calK \subseteq \widetilde \calK$
and for any $n \geq N$ and $t \in \ccint{0,1}$, $\vartheta_n +
t(\vartheta_{n+1}-\vartheta_n) \in \widetilde \calK$. By \ref{lemma:rappel:AMP:item2:hyp5}-\ref{lemma:rappel:AMP:item2:hyp5b},
there exists a
constant $C$ such that for $\vartheta,\vartheta' \in \widetilde \calK$, $|\nabla w(\vartheta) - \nabla
w(\vartheta') | \leq C |\vartheta-\vartheta'|$ showing that, for $n \geq N$,
  \begin{align*}
    w(\vartheta_{n+1})
    & \leq w(\vartheta_n) + \pscal{\nabla w(\vartheta_n)}{\vartheta_{n+1} -
      \vartheta_n} + C/2 \left| \vartheta_{n+1} - \vartheta_n \right|^2.
  \end{align*}
Using \eqref{eq:definition:vartheta}, we obtain
\begin{align*}
  & \pscal{\nabla w(\vartheta_n)}{\vartheta_{n+1} - \vartheta_n} = \rho_{n+1} \pscal{\nabla
    w(\vartheta_n)}{h(\vartheta_n)} + \rho_{n+1} \pscal{\nabla w(\vartheta_n)}{ \xi_{n+1}} \\
  & | \vartheta_{n+1} - \vartheta_n |^2 \leq 2 \rho^2_{n+1} \left\{ \left( \sup_\calK |h| \right)^2 + |\xi_{n+1}|^2  \right\}.
\end{align*}
This yields for any $n \geq N$,
\begin{multline*}
  w(\vartheta_{n+1}) \leq w(\vartheta_n) - \rho_{n+1} \left| \pscal{\nabla w(\vartheta_n)}{h(\vartheta_n)}
  \right| + \rho_{n+1} \pscal{\nabla w(\vartheta_n) }{\xi_{n+1}} \\
  + C \, \rho_{n+1}^2 \left(\left( \sup_\calK |h| \right)^2 + |\xi_{n+1}|^2
  \right) \eqsp.
\end{multline*}
\Cref{lemma:RobbinsSiegmund:deterministe} concludes the proof.
\end{proof}
\subsection{Regularity in $\param$  of the solution to the Poisson equation}
\label{sec:poisson}
Under the assumptions \Cref{hyp:PhiandH} and \Cref{hyp:noyau}, for any $\param
\in \thetaset$, there exists a function $g$ solving the Poisson equation
\begin{equation}
  \label{eq:poisson:equation}
g \mapsto   \H{\param}{} - \pi_\param \H{\param}{} = g - P_\param g
\eqsp.
\end{equation}
This solution, denoted by $g_\param$, is unique up to an additive constant and
given by $g_\param(x) \eqdef \sum_{n \geq 0} \left\{ P_\param^n \H{\param}{x} -
  \pi_\param \H{\param}{} \right\}$.  Finally, for any compact set $\calK$ of
$\thetaset$,
\begin{equation}\label{eq:gtheta:controlK}
\sup_{\param \in \calK} \left| g_\param \right|_W \leq \sup_{\param \in \calK} | \H{\param}{}|_W \
\sum_{n \geq 0} \sup_{x\in\X} \frac{\sup_{\param \in \calK} \| P_\param^n(x,\cdot) - \pi_\param
\|_W}{W(x)} \, < \infty \eqsp.
\end{equation}
\begin{lemma} \label{lemme:ecart:puissance}
  Assume \Cref{hyp:noyau}-\ref{hyp:noyau:ergo:geom}. For any compact set
  $\calK \subset \thetaset$, there exists a constant $C$ such that for any
  $\param, \param' \in \calK$
  \begin{align*}
  \sup_{n \geq 0} \sup_{x \in \X} \frac{\|P_{\param}^n(x,.) - P_{\param'}^n(x,.)\|_{W}}{W(x)} \leq C D_W(\param, \param') \eqsp,
  \end{align*}
where $D_W$ is defined by \eqref{def:D}.
 \end{lemma}
 \begin{proof}
   For any measurable function $f$ such that $|f|_W \leq 1$, it holds
  \begin{align*}
   P_{\param}^n f(x) - P_{\param'}^n f(x) & = \sum_{j=0}^{n-1} P_{\param'}^j (P_{\param} - P_{\param'}) \left( P_{\param}^{n-j-1} f(x) - \pi_{\param}(f) \right) \eqsp.
  \end{align*}
  For any $0 \leq j \leq n-1$,
  \begin{align*}
   \left| P_{\param'}^j (P_{\param} - P_{\param'}) \left( P_{\param}^{n-j-1} f(x) - \pi_{\param}(f) \right) \right| &\leq P_{\param'}^j W(x) \left|(P_{\param} - P_{\param'}) \left( P_{\param}^{n-j-1} f - \pi_{\param}(f) \right) \right|_W \\
   & \leq  D_W(\param,\param') \ P_{\param'}^j W(x) \left|P_{\param}^{n-j-1} f - \pi_{\param}(f) \right|_W \eqsp.
  \end{align*}
  By \Cref{hyp:noyau}-\ref{hyp:noyau:ergo:geom}, there exist $C>0$ and $\cgeom
  \in \ooint{0,1}$ such that for any $\param, \param'\in \calK$,
\begin{align*}
  P_{\param'}^j W(x) \left|P_{\param}^{n-j-1} f - \pi_{\param}(f) \right|_W
  \leq C \left( \cgeom^j W(x) + \pi_{\param'}(W) \right) \cgeom^{n-j-1} \eqsp.
\end{align*}
This concludes the proof.
 \end{proof}

 \begin{lemma} \label{lemme:ecart:pi}
   Assume \Cref{hyp:noyau}-\ref{hyp:noyau:loi_inv},\ref{hyp:noyau:ergo:geom}.  For any compact
   set $\calK \subset \thetaset$, there exists $C >0$ such that for any $\param, \param' \in \calK$, $ \|
   \pi_{\param} - \pi_{\param'} \|_W \leq C D_W(\param,\param')$.
\end{lemma}
\begin{proof}
  For any $x \in \X, n \in \Nset$,
\begin{align*}
  \| \pi_{\param} - \pi_{\param'}\|_W & \leq \left\| \pi_{\param} -
    \P{}^{n}(x,\cdot) \right\|_W + \left\| \P{}^{n}(x,\cdot) -
    P_{\param'}^{n}(x,\cdot) \right\|_W + \left\| P_{\param'}^{n}(x,\cdot)
    -\pi_{\param'}\right\|_W \eqsp.
\end{align*}
Let $\calK$ be a compact subset of $\thetaset$. By
\Cref{hyp:noyau}-\ref{hyp:noyau:ergo:geom}, there exist constants $C >0$
and $\cgeom \in \ooint{0,1}$ such that for any $n \in \Nset$ and $x \in \X$
$  \sup_{\param \in \calK} \left\| \pi_{\param} - \P{}^{n}(x,\cdot)
  \right\|_W \leq C \cgeom^{n} W(x)$.
Moreover, using \Cref{lemme:ecart:puissance}, there exists a constant $C'
>0$ such that for any $\param, \param' \in \calK$ and any $x \in \X$,
$  \sup_{n \geq 0} \left\| \P{}^{n}(x,\cdot) - P_{\param'}^{n}(x,\cdot) \right\|_W \leq C' D_W(\param,\param') W(x)$.
The proof follows, upon noting that $x$ is fixed and arbitrarily chosen.
\end{proof}

 \begin{lemma} \label{lemme:ecart:h}
   Assume \Cref{hyp:PhiandH},
   \Cref{hyp:noyau}-\ref{hyp:noyau:loi_inv},\ref{hyp:noyau:ergo:geom}
   and \Cref{hyp:H:deltaH}. For any compact set $\calK \subset \thetaset$,
   there exists $C>0$  such that for any $\param,
   \param' \in \calK$,
 \begin{align*}
   \norm{h(\param) - h(\param')} \leq C \left( D_W(\param, \param') + \norm{ \param -
     \param'}^{\alpha} \right) \eqsp,
 \end{align*}
where $D_W$ and $\alpha$ are given by \eqref{def:D} and \Cref{hyp:H:deltaH}.
\end{lemma}
\begin{proof}
  Let $\calK$ be a compact subset of $\thetaset$ and $\param$, $\param'$ in
  $\calK$.  By definition of $h$, it holds
\begin{align*}
  \norm{h(\param)-h(\param') } &= \norm{ \pi_{\param}\H{\param}{}
  -\pi_{\param'}\H{\param'}{} } \leq \pi_{\param}\norm{\H{\param}{}-\H{\param'}{}} +
  \norm{(\pi_{\param} -\pi_{\param'})\H{\param'}{}} \eqsp.
\end{align*}
Condition \Cref{hyp:H:deltaH} implies that there exists a constant $C>0$
such that for any $\param, \param' \in \calK$,
$\pi_{\param}\norm{\H{\param}{}-\H{\param'}{} } \leq C \norm{ \param - \param'
}^{\alpha}$.  By \Cref{lemme:ecart:pi} and \Cref{hyp:PhiandH}, there
exist $C>0$ such that for any $\param, \param' \in \calK$, $ \norm{(\pi_{\param}
-\pi_{\param'})\H{\param'}{}} \leq C D_W(\param,\param')$.  The proof follows.
\end{proof}

For any $\vartheta \in \thetaset$, $x \in \X$ and $L \geq 0$, set
\begin{equation}\label{eq:definition:calH}
\calH_{\vartheta,L}(x) \eqdef \sup_{\{\theta \in \calK: \|\theta- \vartheta
  \| \leq L \}} \norm{ \H{\param}{x} - \H{\vartheta}{x} } \eqsp.
\end{equation}
\begin{proposition}
\label{prop:ContinuitePoisson}
Assume \Cref{hyp:PhiandH},
\Cref{hyp:noyau}-\ref{hyp:noyau:loi_inv},\ref{hyp:noyau:ergo:geom}, and
\Cref{hyp:H:deltaH}.  Let $g_\param$ be the solution of
(\ref{eq:poisson:equation}). For any compact set $\calK \subset \thetaset$,
there exist constants $C>0$ and $\cgeom\in \ooint{0,1}$ such that for any
$\param, \param' \in \calK$, $x \in \X$, $n \geq 1$, $L>0$, and any $\vartheta
\in \calK$ such that $\norm{\param-\vartheta } \leq L$,
  \begin{multline*}
    \norm{ P_\param g_\param (x) - P_{\param'} g_{\param'}(x) } \leq C
    \Big\{ \lambda^n W(x) + 2 \, \sum_{l=1}^n P_\vartheta^l
    \calH_{\vartheta,L}(x)    \\
    + n \, \norm{\param - \param'}^\alpha + n \, D_W(\param, \param') + n \,
    D_W(\param, \vartheta) W(x) + n \, D_W(\param', \vartheta) W(x) \Big\}
    \eqsp.
  \end{multline*}
\end{proposition}
\begin{proof}
  For any $\param, \param' \in \thetaset$ and any $n \geq 1$, we write
\begin{align*}
  P_\param g_\param (x) - P_{\param'} g_{\param'}(x) &= \sum_{l > n} \{
  P_\param^l \H{\param}{x} - h(\param) \} - \sum_{l > n} \{ P_{\param'}^l
  \H{\param'}{x} - h(\param') \} + \psi \left( h(\param') - h(\param) \right) \\
  &+ \sum_{l=1}^n \{ P_\param^l \H{\param}{x} - P_{\param'}^l \H{\param'}{x}
  \} \eqsp.
\end{align*}
We first prove that
\begin{multline}
  \norm{ P_\param g_\param (x) - P_{\param'} g_{\param'}(x) - \sum_{l=1}^n
    \{ P_\param^l \H{\param}{x} - P_{\param'}^l \H{\param'}{x} \} } \\
    \leq C \left( \lambda^n W(x) + n \, D_W(\param, \param') + n \, \norm{\param -
      \param'}^\alpha \right) \eqsp. \label{eq:ContinuitePoisson:step1}
\end{multline}
By \Cref{hyp:PhiandH} and \Cref{hyp:noyau}-\ref{hyp:noyau:ergo:geom},
for any compact set $\cal K$, there exist constants $C>0$ and $\lambda \in
\ooint{0,1}$ such that for any $x,n$
\[
\sup_{ \param \in \calK} \norm{ \sum_{l > n} \{ P_\param^l \H{\param}{x} -
  h(\param) \} } \leq C \lambda^{n+1} \sup_{\param \in \calK}
|\H{\param}{}|_W \, W(x) \eqsp.
\]
(\ref{eq:ContinuitePoisson:step1}) follows from this inequality and
\Cref{lemme:ecart:h}.  We now establish an upper bound for $\norm{
  \sum_{l=1}^n \{ P_\param^l \H{\param}{x} - P_{\param'}^l \H{\param'}{x} \}
}$; we first write
\[
\sum_{l=1}^n \norm{ P_\param^l \H{\param}{x} - P_{\param'}^l \H{\param'}{x} }
\leq \sum_{l=1}^n \norm{ P_\param^l \H{\param}{x} - P_{\vartheta}^l
  \H{\vartheta}{x} } + \sum_{l=1}^n \norm{ P_\vartheta^l \H{\vartheta}{x} -
  P_{\param'}^l \H{\param'}{x} } \eqsp.
\]
For any $l \geq
1$ and $\vartheta \in \calK$ such that $\norm{\theta- \vartheta } \leq L$, we have
\begin{align*}
  \norm{ P_\param^l \H{\param}{x} - P_\vartheta^l \H{\vartheta}{x}}& \leq
  P_\vartheta^l \norm{ \H{\param}{}-\H{\vartheta}{}}(x) + \norm{\left(
      P_\param^l - P_\vartheta^l \right) \H{\param}{x} } \\
  & \leq P_\vartheta^l \calH_{\vartheta,L}(x) + \norm{\left( P_\param^l -
      P_\vartheta^l \right) \H{\param}{x} } \eqsp.
\end{align*}
By \Cref{lemme:ecart:puissance}, the second term is upper bounded by $C
D_W(\param, \vartheta) W(x)$ for a constant $C$ depending upon $\calK$ (and
independent of $L$ and $l$).  This concludes the proof.
\end{proof}

\subsection{Proof of \Cref{theorem:general}}
\label{sec:proof:theorem:general}
Define the shifted sequence
\begin{equation}
  \label{eq:TranslatedGamma}
   {\bs \gamma^{\leftarrow q}} = \{\gamma_{q+n}, n \in \Nset \} \eqsp;
\end{equation}
and for any measurable set $\cal K$ of $\thetaset$, define the exit-time from
$\calK$
 \begin{align}
\label{eq:def:calK}
\sigma(\calK) = \inf\{n \geq 1, \param_n \notin \calK\} \eqsp,
\end{align}
with the convention that $\inf \emptyset = +\infty$. If $I_N=I_{N-1}+1 =i$ i.e.
the $i$-th update of the active set occurs at iteration $N$ then
\begin{align}\label{eq:Iteration:init}
  X_{N+1} \sim P_{\param_\star}(x_\star, \cdot)\eqsp, \qquad \param_{N+1} =
  \theta_\star + \gamma^{\leftarrow i}_{1}
  \H{\theta_\star}{X_{N+1}} \eqsp,
\end{align}
and for any $\zeta \geq 1$, while $\param_{N+\zeta} \in \calK_{I_N}$,
\begin{align*}
  X_{N+\zeta+1} \sim P_{\param_{N+\zeta}}(X_{N+\zeta}, \cdot) \qquad
  \param_{N+\zeta+1} = \param_{N+\zeta} + \gamma^{\leftarrow i}_{\zeta+1}
  \H{\param_{N+\zeta}}{X_{N+\zeta+1}} \eqsp.
\end{align*}
This iterative scheme can be seen as a perturbation of the algorithm
$\tau_{N+\zeta+1} = \tau_{N+\zeta} + \gamma^{\leftarrow i}_{\zeta+1}
h(\tau_{N+\zeta})$ and we will show that the sequence $\{\param_n,n \geq 0 \}$
converges as soon as the perturbations $\{\H{\param_k}{X_{k+1}} - h(\param_k),
k \in \Nset \}$ are small enough in some sense. We therefore preface the proof
of \Cref{theorem:general} by preliminary results on the control of
\begin{align*}
  S_{k,l}(\bs \rho, \calK) \eqdef \indic{l \leq \sigma(\cal K)} \sum_{j=k}^{l}
  \rho_j \A_{\param_{j-1}} \left\{ \H{\param_{j-1}}{X_j} - h(\param_{j-1})\right\} \eqsp,
\end{align*}
for $l \geq k \geq 1$, a stepsize sequence $\bs \rho = \suite{\rho}$, a compact
subset $\calK$ of $\thetaset$, and $\sigma(\calK)$ defined by
(\ref{eq:def:calK}). Let $\param \mapsto \A_\param$, $\param \in \Theta$, be a measurable
$d' \times d$ matrix function, where $d' \geq 1$; we will apply the result to
$\A_\theta= \mathrm{I}_{d \times d}$ where $\mathrm{I}_{d \times d}$ is the $d \times d$ identity matrix and $\A_{\param}= \nabla w(\param)'$.

For a sequence $\bs \rho= \suite{\rho}$, denote by ${\PP}_{x,\param}^{\bs
  \rho}$ (resp.  ${\PE}_{x,\param}^{\bs \rho}$) the probability (resp. the
expectation) associated with the non-homogeneous Markov chain on $\X \times
\thetaset$ with $\delta_{(x,\param)}$ as initial distribution and with
transition mechanism given by line \ref{line:algo1:updateX} and line
\ref{line:algo1:updateT} of \Cref{algo:SA:basique}:
\[
X_{n+1} \sim P_{\param_n}(X_n, \cdot) \qquad \param_{n+1} = \param_n +
\rho_{n+1} \H{\param_n}{X_{n+1}} \eqsp.
\]

\begin{proposition}\label{prop:controle:Sln}
  Assume \Cref{hyp:PhiandH}, \Cref{hyp:noyau}, \Cref{hyp:Wfluctuation} and
  \Cref{hyp:H:deltaH}.  Let $\bs \rho$ be a non-increasing positive sequence,
  $\bs \psi = \{\psi_n, n \in \Nset \}$ be a sequence such that $1 \leq \psi_n
  \leq n$ and $\calK$ be a compact set of $\thetaset$.  Let $\param \mapsto
  \A_\param$, $\param \in \Theta$ be a $(d' \times d)$-matrix valued function
  such that $\sup_{\theta \in \calK} \norm{\A_\theta} < \infty$ and
  $\sup_{\theta, \theta' \in \calK} \norm{\theta - \theta'}^{-1} \ 
  \norm{\A_\theta - \A_{\theta'}} \leq C_\A $. Then, there exists a constant
  $C$ such that for any $\delta>0$, $r \in \ocint{0,1}$ and any $(x,\param) \in
  \X \times \calK$,
  \begin{multline*}
    \PP_{x,\param}^{\bs \rho} \left( \sup_{l \geq k} \norm{ S_{k,l}\left( \bs
          \rho, \calK \right) } \geq \delta \right)  \leq \delta^{-1} C \sum_{j \geq k} \rho_{j-\psi_j}^{1+r\alpha} \psi_j^{1+\alpha} + \delta^{-1} C  W^p(x) \sum_{j \geq k} \rho_j^{p(1-r)} \\
    + \delta^{-1} W(x) \left\{ \rho_k + \sum_{j\geq k} \left( \rho_j
        \lambda^{\psi_j} + \rho_{j-\psi_j}^{1+r \pa} \psi_j^3 \right) \right\} + C_\A \, C \delta^{-1} W^2(x) \sum_{j \geq k} \rho_j^2 \eqsp.
  \end{multline*}
\end{proposition}
\begin{proof}
  The proof is postponed to \Cref{sec:proof:prop:controle:Skl}.
\end{proof}

\begin{corollary}
  \label{coro:controle:Sln}  Assume \Cref{hyp:PhiandH} to \Cref{hyp:sommes:pas:poly}.   Let $\calK$ be a compact set of  $\thetaset$ and $(x_\star, \param_\star) \in \X \times \calK$. For any $\delta>0$ and any $\varepsilon \in
  \ooint{0,1}$, there exists $i_\star$ such that for any $i \geq i_\star$
\[
\PP_{x_\star,\param_\star}^{\bs \gamma^{\leftarrow i}}\left( \sup_{n \geq 1} \norm{
    S_{1,n}\left(\bs \gamma^{\leftarrow i}, \calK\right)} \geq \delta
\right) \leq \varepsilon \eqsp.
\]
For any $\delta>0$ and any $i \geq 0$,
\[
\lim_{k \to \infty} \PP_{x_\star,\param_\star}^{\bs \gamma^{\leftarrow i}} \left( \sup_{l
    \geq k} \norm{ S_{k,l}\left( \bs \gamma^{\leftarrow i}, \calK \right)
  }\geq \delta \right) = 0 \eqsp.
\]
\end{corollary}
\begin{proof}
  By \Cref{hyp:PhiandH}, $W(x_\star)< \infty$.  Let $\alpha, \beta$ be resp.
  given by \Cref{hyp:H:deltaH} and \Cref{hyp:sommes:pas:poly}.  We apply
  \Cref{prop:controle:Sln} with $r \in \ooint{(\alpha \wedge \pa)^{-1}
    (\beta^{-1}-1); 1 - (\beta p)^{-1}}$, $\bs \rho = \bs \gamma^{\leftarrow
    i}$ and $\psi_j \sim \tau \ln j$ when $j \to \infty$ for some $\tau \in
  \ooint{(\beta-1)/\ln \cgeom,+\infty}$.
\end{proof}

\paragraph{Proof of \Cref{theorem:general}}
Define the sequence of exit-times
\[
T_0 = 0 \qquad T_m = \inf\{n \geq T_{m-1}+1, I_n = I_{n-1}+1 \}, m \geq 1
\eqsp.
\]
\begin{enumerate}[label=(\roman*), wide=0pt, labelindent=\parindent]
\item Let $M_0$ be such that $\calL \cup \calK_0 \subset \{\param \in \Theta:
  w(\param) \leq M_0 \}$ and set $\calW \eqdef \{\param \in \Theta: w(\param)
  \leq M_0+1 \}$.  Since $\bs \gamma$ is decreasing,
  \Cref{lemma:rappel:AMP:item1} shows that there exist $\delta_\star>0$ and
  $i_\star \geq 0$ large enough such for any $i \geq i_\star$,
\[
\PP_{x_\star,\param_\star}^{\bs \gamma^{\leftarrow i}}\left( \sigma\left( \calW\right)<
  \infty \right) \leq \PP_{x_\star,\param_\star}^{\bs \gamma^{\leftarrow i}}\left( \sup_{n
    \geq 1} \norm{S_{1,n}\left(\bs \gamma^{\leftarrow i}, \calW\right) } >
  \delta_\star \right) \eqsp.
\]
\Cref{coro:controle:Sln}  shows that for any $\varepsilon \in \ooint{0,1}$,
there exists $j_\star$ such that for any $j \geq j_\star$
\[
\PP_{x_\star,\param_\star}^{\bs
  \gamma^{\leftarrow j}}\left( \sup_{n \geq 1} \norm{ S_{1,n}\left(\bs
      \gamma^{\leftarrow j}, \calW\right) } > \delta_\star \right) \leq
\varepsilon \eqsp.
\]
We can assume w.l.o.g. that $j_\star = i_\star$ and we do so. On the other
hand, since $\calW$ is a compact subset of $\thetaset$,
by~(\ref{eq:ConditionCompact}), there exists $m_\star$ such that for any $m
\geq m_\star$, $\calW \subset \calK_{m}$. Hereagain, we can assume that
$i_\star = m_\star$ and we do so.  Hence, for any $i \geq i_\star$
\begin{equation}
  \label{eq:contraction:retour}
 \PP_{x_\star,\param_\star}^{\bs
  \gamma^{\leftarrow i}}\bigg( \sigma\left( \calW\right)< \infty \bigg) \leq
\varepsilon \eqsp.
\end{equation}
This yields for all $i \geq i_\star$,
\begin{align*}
  \overline{\PP}_{x_\star,\param_\star,0} \left( T_{i+1} < \infty \right) & =
  \overline{\PE}_{x_\star,\param_\star,0} \left[ \indic{T_{i} < \infty} \
    \PP_{x_\star,\param_\star}^{\bs \gamma^{\leftarrow i}} \left(
      \sigma\left( \calK_i\right) < \infty \right) \right] \\
  & \leq \overline{\PP}_{x_\star,\param_\star,0} \left( T_{i} < \infty \right)
  \PP_{x_\star,\param_\star}^{\bs \gamma^{\leftarrow i}} \left( \sigma\left(
      \calW\right) < \infty \right) \leq \varepsilon^{i-i_\star} \eqsp,
\end{align*}
where we used (\ref{eq:contraction:retour}) and a trivial induction in the last
inequality. Since $\varepsilon \in \ooint{0,1}$, we have $\sum_i
\overline{\PP}_{x_\star,\param_\star,0} \left( T_{i+1} < \infty \right) <
\infty$ which yields $\overline{\PP}_{x_\star,\param_\star,0}(\limsup_i \{T_i <
\infty \}) =0$ by the Borel-Cantelli lemma.
\item By \Cref{theorem:general}(i), $I \eqdef \sup_n I_n$ is finite
  $\overline{\PP}_{x_\star,\param_\star,0}$-\as\ and
  $\overline{\PP}_{x_\star,\param_\star,0}(\forall n \geq 1, \param_n \in
  \calK_I) =1$.  Since $I$ is finite \as , it is equivalent to prove that for
  any $i \geq 0$, on the set $\{I=i \}$, $\lim_{k} \mathrm{d}(\param_k, \calL)
  = 0$ \as\ Let $i$ be fixed. We apply \Cref{lemma:rappel:AMP:item2} with $\bs
  \rho = \bs \gamma^{\leftarrow i}$, $\vartheta_k = \param_{T_{i}+k}$ and
  $\xi_j \leftarrow \H{\param_{j}}{X_{j+1}} -h\left(\param_{j}
  \right)$.
  \\
  $h$ is Holder-continuous under \Cref{hyp:noyau}, \Cref{hyp:Wfluctuation} and
  \Cref{hyp:H:deltaH}.  Since $\sum_k \gamma_k = +\infty$ and $\lim_k
  \gamma_k=0$ by assumptions, then $\sum_k \rho_k = \infty$ and $\lim_k \rho_k
  = 0$.  For any $\delta >0$, by applying the strong Markov property with the
  stopping-time $T_i$, we have
\begin{align*}
  & \overline{\PP}_{x_\star,\param_\star,0}\left( \limsup_k \sup_{l \geq k}
    \norm{ \sum_{j=k}^l \gamma_{i+j+1} \A_{\param_{T_i+j}} \left\{
        \H{\param_{T_i+j}}{X_{T_i+j+1}} -h\left(\param_{T_i+j} \right)\right\} } \geq \delta, I=i \right) \\
  & \leq {\PP}_{x_\star,\param_\star}^{\bs \gamma^{\leftarrow i}} \left(
    \limsup_k \sup_{l \geq k} \norm{ \sum_{j=k}^l \gamma_{i+j+1}
      \A_{\param_{T_i+j}} \left\{ \H{\param_{j}}{X_{j+1}} -h\left(\param_{j}
        \right)\right\} } \geq
    \delta, \sigma(\calK_i) = +\infty\right) \\
  & \leq {\PP}_{x_\star,\param_\star}^{\bs \gamma^{\leftarrow i}} \left(
    \limsup_k \sup_{l \geq k} \norm{ S_{k,l}\left( \bs \gamma^{\leftarrow i},
        \calK_i \right) }\geq
    \delta \right) \\
  & \leq \lim_k {\PP}_{x_\star,\param_\star}^{\bs \gamma^{\leftarrow i}} \left(
    \sup_{l \geq k} \norm{ S_{k,l}\left( \bs \gamma^{\leftarrow i}, \calK_i
      \right) } \geq \delta \right) \eqsp.
\end{align*}
The RHS is zero by \Cref{coro:controle:Sln}.  Hence, by choosing
$\A_{\vartheta_j} = \mathrm{I}_{d \times d}$, the
\Cref{lemma:rappel:AMP:item2}-\ref{lemma:rappel:AMP:item2:hyp4} holds; note
also that with this choice of $\A_\theta$, $C_\A = 0$.  Let us check
\Cref{lemma:rappel:AMP:item2}-\ref{lemma:rappel:AMP:item2:hyp5}.  Under
\Cref{theorem:general}-\ref{theorem:general:hyp1surw},
\Cref{lemma:rappel:AMP:item2}-
\ref{lemma:rappel:AMP:item2:hyp5}-\ref{lemma:rappel:AMP:item2:hyp5a} holds.

Assume now that \Cref{theorem:general}-\ref{theorem:general:hyp2surw} is satisfied; we
prove that \Cref{lemma:rappel:AMP:item2}-\ref{lemma:rappel:AMP:item2:hyp5}-\ref{lemma:rappel:AMP:item2:hyp5b} holds.
Along the same lines as above, and
choosing $\A_{\theta_j}$ equal to the transpose of $\nabla w(\theta_j)$, we
establish that $\lim_k \sum_{j=1}^k \rho_j \pscal{\nabla w(\vartheta_j)}{
  \xi_j}$ exists.  By \Cref{hyp:PhiandH} and since $\sup_\calK \norm{h} <
\infty$, there exists a constant $C$ such that
 \begin{align*}
  & \overline{\PP}_{x_\star,\param_\star,0} \left( \limsup_k \sup_{l \geq k}
     \sum_{j=k}^l \gamma_{i+j+1}^2 \norm{ \H{\param_{T_i+j}}{X_{T_i+j+1}}
       -h\left(\param_{T_i}\right)}^2 \geq \delta, I=i \right) \\
     & \leq \lim_k {\PP}_{x_\star,\param_\star}^{\bs \gamma^{\leftarrow i}}
     \left( \sup_{l \geq l} \sum_{j=k}^k \gamma^{\leftarrow i}_{j+1}
       (W^2(X_{j+1}) +1) \geq \delta / C \right) \eqsp.
 \end{align*}
 The RHS tends to zero since $\sup_j \PE_{x_\star,\param_\star}^{\bs
   \gamma^{\leftarrow i}} \left[W^2(X_j) \right] < \infty$ (we assumed $p \geq
 2$ in \Cref{hyp:noyau}-\ref{hyp:noyau:control:W:traj}) and $\sum_j \gamma_j^2
 < \infty$ (we assumed that $\beta >1/2$).  Hence, \Cref{lemma:rappel:AMP:item2}-\ref{lemma:rappel:AMP:item2:hyp5}-\ref{lemma:rappel:AMP:item2:hyp5b}
 holds.

We then conclude by \Cref{lemma:rappel:AMP:item2} that
$\overline{\PP}_{x_\star, \param_\star,0}$-\as\, on the set ${I=i}$, the
sequence $\{\param_k, k \geq 0\}$ converges to a connected component of
$\calL$.
\end{enumerate}

\subsection{Proof of \Cref{prop:controle:Sln}}
\label{sec:proof:prop:controle:Skl}
Let $\bs \rho$ be a non-increasing positive sequence and $\calK$ be a compact
subset of $\thetaset$ such that $\calK_0 \subseteq \calK$.  Throughout this
section, $\bs \rho$ and $\calK$ are fixed; we will therefore use the notations
$S_{k,l}$ and $\sigma$ instead of $S_{k,l}(\bs \rho, \calK)$ and
$\sigma(\calK)$. Set
\begin{equation} \label{eq:definition:Cstar}
C_\star \eqdef \sup_{\param \in \calK} \{ |g_\param|_W + |P_\param g_\param|_W \} \vee \sup_{\param \in \calK} \norm{\A_\param} \eqsp,
\end{equation}
where $g_\param$ is the solution to the Poisson
equation~(\ref{eq:poisson:equation}). $C_\star$ is finite by
(\ref{eq:gtheta:controlK}), \Cref{hyp:PhiandH} and
\Cref{hyp:noyau}-\ref{hyp:noyau:ergo:geom}.  By~\eqref{eq:poisson:equation},
$\H{\param_{j-1}}{X_j} - h(\param_{j-1}) = \g{j-1}(X_j) - \P{j-1}\g{j-1}(X_j)$
for any $j \geq 1$.  We then write $ S_{k,l} = \indic{\sigma \geq l}
\sum_{i=1}^4 T_{k,l}^{(i)}$ with
 \begin{align*}
   T_{k,l}^{(1)} & \eqdef \sum_{j=k}^l \rho_j \A_{\theta_{j-1}} \left(\g{j-1}(X_j) - \P{j-1}\g{j-1}(X_{j-1}) \right) \indic{j \leq  \sigma} \eqsp, \\
   T_{k,l}^{(2)}& \eqdef \sum_{j=k}^{l} \rho_{j+1} \A_{\theta_{j}} \left(\P{j}\g{j}(X_{j})  - \P{j-1}\g{j-1}(X_{j}) \right) \indic{j+1 \leq \sigma} \eqsp, \\
   T_{k,l}^{(3)}& \eqdef \sum_{j=k}^{l}   \left( \rho_{j+1} \A_{\theta_{j}}  \indic{j+1 \leq \sigma} -  \rho_j \A_{\theta_{j-1}} \indic{j \leq \sigma} \right) \P{j-1}\g{j-1}(X_j) \eqsp, \\
   T_{k,l}^{(4)} & \eqdef \rho_k \A_{\theta_{k-1}} \P{k-1}\g{k-1}(X_{k-1}) -
   \rho_{l+1} \A_{\theta_{l}} \P{l}\g{l}(X_{l}) \indic{l+1 \leq \sigma} \eqsp.
 \end{align*}
For any measurable set $\calA$,  we can bound by Markov's and Jensen's inequality
\begin{align*}
  \PP_{x,\param }^{\bs \rho}\left(\sup_{l\ge k} |S_{k,l}|\ge \delta \right)
  &\le \frac{3}{\delta}\Bigg( \PE_{x,\param }^{\bs \rho}\left[\sup_{l\ge k}
    \left|T_{k,l}^{(1)}\right|^p\right]^{1/p} + \PE_{x,\param }^{\bs
    \rho}\left[\sup_{k\le l\le \sigma}
    \left|T_{k,l}^{(3)} + T_{k,l}^{(4)}\right|\right] \\
  &\phantom{\le\frac{1}{\delta}} + \PE_{x,\param }^{\bs
    \rho}\left[\sup_{k\le l\le \sigma} \left|T_{k,l}^{(2)}\indic{\calA}\right|
  \right] \Bigg) + \PP_{x,\param }^{\bs \rho}(\calA^c)\eqsp .
\end{align*}
The terms on the right are bounded individually by the three lemmas below,
concluding the proof of \Cref{prop:controle:Sln}.

 \begin{lemma}
   Assume \Cref{hyp:PhiandH}, \Cref{hyp:noyau} and $\sup_{\theta \in \calK}
   \norm{\A_\theta} < \infty$.  There exists a constant $C$
   - which does not depend on $\bs \rho$ - such that for any $x \in \X$,
   $\param \in \calK_0$ and $k \geq 1$,
 \[
 \PE_{x,\param }^{\bs \rho}\left[ \sup_{l \geq k} \norm{ T_{k,l}^{(1)}}^p
 \right] \leq C \, W^p(x) \sum_{\ell \geq k} \rho_\ell^p \eqsp.
 \]
 \end{lemma}
 \begin{proof}
   Let $k \geq 1$ be fixed.  Note that $\{T_{k,l}^{(1)}, l \geq k \}$ is a
   $\F_l$-martingale under the probability $\PP_{x,\param}^{\bs \rho}$ which
   implies that for any $p>1$, there exists a constant $C$ such that (see \eg\
   \cite[Theorems 2.2 and 2.10]{hall:heyde:1980})
  \begin{align*}
    \PE_{x,\param}^{\bs \rho} \left[ \sup_{l \geq k} \norm{ T_{k,l}^{(1)} }^p
    \right] &\leq \lim_{L \to \infty} \PE_{x,\param}^{\bs \rho} \left[
      \PE_{x,\param}^{\bs \rho} \left[ \sup_{k \leq l \leq k+L}
        \norm{T_{k,l}^{(1)}}^p \Big\vert \F_k \right] \right]  \\
    &\leq C \lim_{L \to \infty} \PE_{x,\param}^{\bs \rho} \left[
      \PE_{x,\param}^{\bs \rho}\left[ \norm{ T_{k,k+L}^{(1)} }^p \Big \vert
        \F_k \right] \right] \\
    &\leq C \, W^p(x) \ \left(\sum_{l \geq k} \rho_l^p \right) \eqsp.
  \end{align*}
 \end{proof}
 \begin{lemma}
   Assume \Cref{hyp:PhiandH}, \Cref{hyp:noyau}, $\sup_{\theta \in \calK}
   \norm{\A_\theta} < \infty$ and $\sup_{\theta , \theta' \in \calK}
   \norm{\theta - \theta'}^{-1} \norm{\A_\theta - \A_{\theta'}} \leq C_\A$.
   There exists a constant $C$ - which does not depend on $\bs \rho$ - such
   that for any $x \in \X$, $\param \in \calK_0$, $k \geq 1$,
 \begin{align*}
   \PE_{x, \param}^{\bs \rho} \bigg[ \sup_{k \leq l \leq \sigma} \norm{
     T_{k,l}^{(3)} + T_{k,l}^{(4)} } \bigg] \leq C \, \left( \rho_k W(x) + C_\A
     \sum_{j \geq k} \rho_j^2 W^2(x)\right) \eqsp.
 \end{align*}
 \end{lemma}
 \begin{proof}
   By (\ref{eq:definition:Cstar}), we have $\left| \P{i} \g{i} \right|_W
   \indic{i< \sigma} \leq C_\star$ and $\sup_{\theta \in \calK}
   \norm{\A_\theta} \leq C_\star$.  Since $\bs \rho$ is non-increasing, this
   yields $\PE_{x, \param}^{\bs \rho} \bigg[ \sup_{k \leq l \leq \sigma} \norm{
     T_{k,l}^{(4)} } \bigg] \leq 2 C_\star^2 \rho_k W(x)$.  We write
   $T_{k,l}^{(3)} = T_{k,l}^{(3,a)} - T_{k,l}^{(3,b)}$ with
\begin{align*}
 &  T_{k,l}^{(3,a)} \eqdef \sum_{j=k}^l \left(\rho_{j+1} \A_{\theta_j}- \rho_j \A_{\theta_{j-1}} \right)
  P_{\theta_{j-1}} g_{\theta_{j-1}}(X_j) \indic{j+1 \leq \sigma} \eqsp, \\
  & T_{k,l}^{(3,b)} \eqdef \sum_{j=k}^l \rho_j  \A_{\theta_{j-1}} P_{\theta_{j-1}}
  g_{\theta_{j-1}}(X_j) \indic{\sigma=j} \eqsp.
\end{align*}
Note that $\{j = \sigma \} \cap \{l \leq \sigma\} = \emptyset$ for any $j < l$. Hence
\[
\PE_{x, \param}^{\bs \rho} \bigg[ \sup_{k \leq l \leq \sigma} \norm{
  T_{k,l}^{(3,b)} } \bigg] = \PE_{x, \param}^{\bs \rho} \bigg[ \rho_l \norm{A_{\theta_{l-1}}} \norm{
  P_{\param_{l-1}} g_{\param_{l-1}}(X_l) } \indic{ l = \sigma} \bigg] \leq
C_\star^2 \rho_k W(x) \eqsp,
\]
where in the inequality we used that $\bs
\rho$ is non-increasing.  Finally, along the same lines, we get
\begin{align*}
  |\rho_{j+1} \A_{\theta_j} - \rho_j \A_{\theta_{j-1}}| \norm{P_{\param_{j-1}} g_{\param_{j-1}}(X_j) }
  \indic{j+1 \leq \sigma} & \leq (\rho_j -\rho_{j+1}) C_\star^2 W(X_j) \indic{j
    <
    \sigma}  \\
  & + \rho_j \norm{\A_{\theta_j} -\A_{\theta_{j-1}}} C_\star W(X_j) \un_{j <
    \sigma}
\end{align*}
Since $\norm{\A_{\theta_j} -\A_{\theta_{j-1}}} \un_{j < \sigma} \leq C_\A
C_\star \rho_j W(X_j)$, this yields $\PE_{x, \param}^{\bs \rho} \bigg[ \sup_{k
  \leq l \leq \sigma} \norm{ T_{k,l}^{(3,a)} } \bigg] \leq C_\star^2 \left(
  \rho_k W(x)+ C_\A \rho_k^2 W^2(x)\right)$.
\end{proof}

For any $r \in \ooint{0,1}$ and $0 \leq k - \psi_k \leq n$, set
 \[
 \calA_r (k,n) \eqdef \bigcap_{\ell=k-\psi_k}^n \left\{ \norm{\param_\ell - \param_{\ell-1}
   } \leq \rho_\ell^r\right\} = \left\{ \sup_{k-\psi_k \leq \ell \leq n} \frac{ \norm{\param_\ell
     - \param_{\ell-1} }}{\rho_\ell^r} \leq 1 \right\} \eqsp.
 \]
 \begin{proposition} \label{prop:terme:continuite}
   Assume \Cref{hyp:PhiandH}, \Cref{hyp:noyau}, \Cref{hyp:Wfluctuation},
   \Cref{hyp:H:deltaH} and $\sup_{\theta \in \calK} \norm{\A_\theta} < \infty$.
   \begin{enumerate}[label=(\roman*)]
   \item \label{prop:terme:continuite:item1} There exists a constant
     $C<\infty$ - which does not depend on $\bs \rho$ - such that for any $r \in
     \ooint{0,1}$, $x \in \X$, $\param \in \thetaset$ and $k \geq
     1$,
 \begin{equation} \label{eq:terme:continuite:item1}
 \PP_{x, \param}^{\bs \rho}\bigg( \calA_{r}(k,\sigma)\bigg) \geq 1 - C \left( \sum_{\ell \geq k-\psi_k} \rho_\ell^{p(1-r)} \right) \ W^p(x)  \eqsp.
 \end{equation}
\item \label{prop:terme:continuite:item2} There exists a constant
  $C>0$ - which does not depend on $\bs \rho$ - such that for any $r \in
  \ooint{0,1}$, $x \in \X$, $\param \in \thetaset$, $k \geq 1$, and
  for any positive sequence $\bs \psi=\{\psi_j, j \in \Nset\}$ such that $1 \leq \psi_j \leq j$,
 \begin{align*}
   C \, \PE_{x, \param}^{\bs \rho} & \left[ \sup_{k \leq l \leq \sigma} \norm{
       T_{k,l}^{(2)}} \indic{\calA_{r}(k,\sigma)}\right]  \\
& \leq \sum_{j \geq k}
   \rho_{j-\psi_j}^{1+r\alpha} \psi_j^{1+\alpha} + W(x) \sum_{j \geq k} \left(
   \rho_{j} \cgeom^{\psi_j} + \rho_{j-\psi_j}^{1+r\pa} \psi_j^3 \right)
 \eqsp.
 \end{align*}
 \end{enumerate}
 \end{proposition}
 \begin{proof}
Throughout this proof, $C$ denotes a constant which may change upon each appearance and only depends on $\calK$.
\begin{enumerate}[label=(\roman*),  wide=0pt, labelindent=\parindent]
\item By \Cref{hyp:PhiandH}, there exists a constant $C$ such that
   $\PP_{x,\param}^{\bs \rho}$-\as\ , on the set $\{ k - \psi_k \leq \ell \leq \sigma\}$,
   $\norm{ \theta_\ell - \theta_{\ell-1}} \leq C \rho_\ell W(X_\ell)$.  Hence, by the Markov inequality
 \begin{align*}
   1 - \PP_{x, \param}^{\bs \rho}\bigg( \calA_r(k,\sigma)\bigg) & \leq C
   \PE_{x, \param}^{\bs \rho}\left[ \sup_{k-\psi_k \leq \ell \leq \sigma}
     \rho_\ell^{p(1-r)} W^p(X_\ell)  \right] \\
   & \leq C \left( \sum_{\ell \geq k-\psi_k } \rho_\ell^{p(1-r)} \right) W^p(x)
   \eqsp.
 \end{align*}
\item We use \Cref{prop:ContinuitePoisson} with $\param \leftarrow \param_j$, $\param'
 \leftarrow \param_{j-1}$, $\vartheta \leftarrow \param_{j- \psi_j}$, $x
 \leftarrow X_j$, $n \leftarrow \psi_j$  and $L \leftarrow L_j \eqdef \psi_j \rho_{j
   -\psi_j+1}^r$.  Since $\bs \rho$ is non-increasing, observe that
 \[
 \left( \norm{\param_j - \param_{j -\psi_j} } \vee \norm{\param_{j-1} - \param_{j-\psi_j} } \right) \indic{\calA_r(k,j)}  \leq \psi_j \,
 \rho_{j-\psi_j+1}^r \eqsp,
 \]
 thus justifying that with the above definitions, we have $\norm{\param -
   \vartheta } \vee \norm{\param'-\vartheta } \leq L$.  By using $D_W(\param,
 \param'') \leq D_W(\param, \param') + D_W(\param', \param'')$ and $W \geq 1$,
 we write $ \sup_{k \leq l \leq \sigma} \norm{ T_{k,l}^{(2)} }
 \indic{\calA_r(k,\sigma)} \leq C \, \sum_{i=1}^2 \Xi_k^{(i)}$ with
\begin{align*}
  \Xi_k^{(1)} &\eqdef \sum_{j \geq k} \rho_{j} \psi_j  \A_{\theta_{j-1}} \left\{ \norm{\param_j -
      \param_{j-1}}^\alpha + 2 D_W(\param_j, \param_{j-1}) W(X_j)  \right. \\
  & \qquad \left. + 2 D_W(\param_{j-1}, \param_{j-\psi_j}) W(X_j)
  \right\}\indic{j \leq \sigma} \indic{\calA_r(k,j)} + \sum_{j \geq k} \rho_{j}  \A_{\theta_{j-1}}
  \cgeom^{\psi_j} W(X_j)
  \indic{j \leq \sigma}  \eqsp, \\
  \Xi_k^{(2)} &\eqdef \sum_{j \geq k} \rho_{j}  \A_{\theta_{j-1}} \sum_{l=1}^{\psi_j}
  P_{\param_{j-\psi_j}}^l \calH_{\param_{j-\psi_j},L_j}(X_j) \indic{j \leq
    \sigma} \indic{\calA_r(k,j-1)} \eqsp.
\end{align*}
Let us consider $\Xi_k^{(1)}$. By (\ref{eq:definition:Cstar}),
\Cref{hyp:Wfluctuation} and the monotonicity of $\bs \rho$, we have
\[
\PE_{x,\param}^{\bs \rho}\left[ \Xi_k^{(1)} \right] \leq C_\star \sum_{j \geq
  k} \rho_j^{1+r \alpha} \psi_j + C_\star W(x) \sum_{j \geq k} \rho_{j}
\cgeom^{\psi_j} + C W(x)\sum_{j \geq k} \psi_j^2 \rho_{j-\psi_j}^{1+r \pa}
\eqsp.
\]
Let us now consider $\Xi_k^{(2)}$. Set $\mathbb{B}_{l,j} \eqdef
\PE_{x,\param}^{\bs \rho} \left[ D_W(\param_{l}, \param_{j-\psi_j}) W(X_{l})
  \indic{l+1 \leq \sigma} \indic{\calA_r(k,l)} \right]$. We write
\begin{align*}
  \PE_{x,\param}^{\bs \rho} & \left[ P_{\param_{j-\psi_j}}^l
    \calH_{\param_{j-\psi_j},L_j}(X_j) \indic{j \leq \sigma}
    \indic{\calA_r(k,j-1)} \right] \\
  & = \PE_{x,\param}^{\bs \rho} \left[ P_{\param_{j-1}}
    P_{\param_{j-\psi_j}}^l \calH_{\param_{j-\psi_j},L_j}(X_{j-1}) \indic{j \leq
      \sigma}
    \indic{\calA_r(k,j-1)} \right] \\
  & \leq \PE_{x,\param}^{\bs \rho} \left[ P_{\param_{j-\psi_j}}^{l+1}
    \calH_{\param_{j-\psi_j},L_j}(X_{j-1}) \indic{j \leq \sigma}
    \indic{\calA_r(k,j-1)} \right] + C \, \mathbb{B}_{j-1,j}\eqsp,
\end{align*}
where in the last inequality, we used that
\[
\left| P_{\param_{j-\psi_j}}^l \calH_{\param_{j-\psi_j},L_j} \right|_W \indic{j
  -\psi_j < \sigma} \leq 2 \sup_{\param \in \calK} \left|
  \H{\param}{}\right|_W \, \, \sup_{l \geq 1} \sup_{\param \in \calK} \left|
  P_{\param}^l W \right|_W
\]
which is finite by \Cref{hyp:PhiandH} and
\Cref{hyp:noyau}-\ref{hyp:noyau:ergo:geom}. Since $\indic{j \leq \sigma}
\indic{\calA_r(k,j-1)} \leq \indic{j-1 \leq \sigma} \indic{\calA_r(k,j-2)}$, we
have by a trivial induction
\begin{equation*}
  \PE_{x,\param}^{\bs \rho} \left[ P_{\param_{j-\psi_j}}^l
    \calH_{\param_{j-\psi_j},L_j}(X_j) \indic{j \leq \sigma}
    \indic{\calA_r(k,j-1)} \right]  \leq \mathbb{A}_{l,j} + C\, \sum_{i=1}^{\psi_j-1} \mathbb{B}_{j-i,j} \eqsp,
\end{equation*}
where $\mathbb{A}_{l,j} \eqdef \PE_{x,\param}^{\bs \rho} \left[ P_{\param_{j-\psi_j}}^{l+\psi_j}
    \calH_{\param_{j-\psi_j},L_j}(X_{j-\psi_j}) \indic{j-\psi_j+1 \leq   \sigma}
    \indic{\calA_r(k,j-\psi_j)} \right]$. Finally, using again $\left|\calH_{\vartheta,L} \right|_W \indic{\vartheta \in \calK} \leq 2 \sup_{\param \in \calK} |\H{\param}{}|_W$ and
\Cref{hyp:noyau}-\ref{hyp:noyau:ergo:geom}, we obtain
\begin{equation*}
\mathbb{A}_{l,j} \leq C \, \cgeom^{l + \psi_j} W(x)+ \sup_{\param \in
    \calK} \pi_\param \calH_{\param,L_j} \eqsp.
\end{equation*}
By \Cref{hyp:H:deltaH}, the last term in the RHS is upper bounded by
$L_j^\alpha$. Combining the above inequalities and using
(\ref{eq:definition:Cstar}), we have
\[
\PE_{x,\param}^{\bs \rho}\left[ \Xi_k^{(2)} \right] \leq C \sum_{j \geq k}
\rho_{j} \psi_j \sum_{i=1}^{\psi_j-1} \mathbb{B}_{j-i,j} + C_\star \sum_{j \geq k} \rho_{j}
\psi_j L_j^\alpha + C \mathbb{W}_{x,\param} \sum_{j \geq k} \rho_{j}
\cgeom^{\psi_j} \eqsp.
\]
The result follows upon noting that $\mathbb{B}_{l,j} \leq W(x)
(l-j+\psi_j) \rho_{j-\psi_j}^{r \pa}$.
\end{enumerate}
 \end{proof}

%% file: proofs-examples.tex
 \subsection{Proof of  \Cref{ex:quantile:d1}}
\label{sec:proof:quantile:d1}
Let $\thetaset = \Rset$. From (\ref{H:quantiles}), $\sup_{\param \in \thetaset}
\sup_{x \in \Rset} \norm{\H{\param}{x}} \leq 1$ so that \Cref{hyp:PhiandH} is
satisfied with the constant function $W = 1$.

We have $ \norm{\H{\param_1}{x} - \H{\param_2}{x}} = \indic{\param_1 \wedge
  \param_2 \leq \phi(x) < \param_1 \vee \param_2}$ and under the stated assumptions,
\[
\sup_{\param \in \Rset} \int \sup_{\param' \in \calB(\param, \delta)}
\norm{\H{\param}{x} - \H{\param'}{x}} \pi(\rmd x) \leq \sup_{\param \in \Rset}
\int \indic{\param - \delta \leq \phi(y) \leq \param + \delta} \pi(y) \rmd y
\leq C \delta \eqsp.
\]
Therefore, \Cref{hyp:H:deltaH} is satisfied with $\alpha=1$.

Since the weak derivative of $\param \mapsto \norm{\param-x}$ is
$\mathrm{sign}(\param-x)$ almost-everywhere, the dominated convergence theorem
implies that $w$ is differentiable and its derivative is
\begin{align} \label{eq:quantile:expressionw'}
  w'({\param}) = \frac{1}{2} \left( \int \un_{\phi(y) \leq \theta} \pi(y) \rmd
    y - \int \un_{\phi(y) \geq \theta} \pi(y) \rmd y \right) + \left(
    \frac{1}{2} - \quant \right) = \int \un_{\phi(y) \leq \theta} \pi(y) \rmd y
  - \quant \eqsp;
\end{align}
we also have $w$ continuously differentiable. Since $\int \norm{\phi(y)} \pi(y)
\rmd y < \infty$,
\begin{align*}
  w(\param) \geq \frac{\norm{\param}}{2} + \left( \frac{1}{2} - \quant \right)
  \param - \frac{1}{2} \int \norm{\phi(y)} \pi(y) \rmd y
  \underset{\norm{\param} \to \infty}{\longrightarrow} \infty \eqsp;
\end{align*}
since $w$ is continuous, this implies that the level sets of $w$ are compact,
thus showing \Cref{hyp:LyapunovFunction}-\ref{hyp:LyapunovFunction:compacity}
holds. By definition of $h$ (see (\ref{eq:meanfield})), we have $h({\param}) =
- w'(\param)$.  Therefore, the set $\calL$ in
\Cref{hyp:LyapunovFunction}-\ref{hyp:LyapunovFunction:calL} is given by
(\ref{L:quantiles}) and it is compact. In addition,
\Cref{hyp:LyapunovFunction}-\ref{hyp:LyapunovFunction:scal} is satisfied.
Finally, $w(\param)$ reaches its minimum at $\param_{\star} \in \calL$ (see
\eqref{eq:quantile:expressionw'}). Since the Lyapunov function $w$ is defined
up to an additive constant, we can assume with no loss of generality that $w$
is non-negative, which concludes the proof of \Cref{hyp:LyapunovFunction}. \\
Note that $w$ is constant on $\calL$ since $w'(\param)=0$ for any $\param \in
\calL$ and $\calL$ is an interval. Hence $w(\calL)$ has an empty interior.

\subsection{Proof of \Cref{prop:convergence-ACE}}
\label{sec:proof:prop:convergence-ACE}
Let $\calK_1$ and $\calK_2$ be resp. a compact of $\Rset$ and $\mathcal{V}$;
set $\calK \eqdef \calK_1 \times \calK_2$.  Let $\tau >0$ be such that $C
\tau^\alpha \leq \delta$ where $C,\delta$ are given by \Cref{hyp:ACE}.  For any
$\vartheta =(\theta,s) \in \calK $, $\vartheta' = (\theta',s')$ with $| \theta-
\theta'| \leq \tau$ and $\norm{s - s'} \leq \tau$, and $x=(y,z) \in \X \times
\X$ it holds
 \begin{align*}
   \norm{\H{\vartheta}{x} - \H{\vartheta'}{x}} & \leq \norm{s- s'} +
   \frac{\mu(z)}{g_{\hat \nu(s)}(z)} \left\{ \norm{\indic{\phi(z) \geq \theta}
       - \indic{\phi(z) \geq \theta'}} + \left| 1 - \frac{g_{\hat
           \nu(s)}(z)}{g_{\hat \nu(s')}(z)}\right| \right\} \\
   & \leq \tau + \frac{\mu(z)}{g_{\hat \nu(s)}(z)} \left\{ \indic{\theta \wedge
       \theta' \leq \phi(z) \leq \theta \vee \theta'} + \psi(s, s', z) \,
     \exp\left( \psi(s, s', z) \right) \right\} \eqsp,
 \end{align*}
 where $\psi(s,s', z) \eqdef \left| B(\hat \nu(s)) - B(\hat \nu(s')) \right| +
 \norm{\hat \nu(s) - \hat \nu(s')} \norm{S(z)}$.  By \Cref{hyp:ACE}, there
 exists a constant $C$ - depending only upon $\calK$ - such that for any
 $\vartheta \in \calK$
\begin{align*}
  \int \pi_{\vartheta}(\rmd x) \sup_{|\theta - \theta'| \leq \tau}
  \norm{\H{\vartheta}{x} - \H{\vartheta'}{x}} \leq \tau + C \tau^\alpha \eqsp.
\end{align*}

\subsection{Proofs of \Cref{sec:quantile:median}}
\label{sec:proof:quantile:median}
\begin{lemma}
 \label{lem:ControlMoment} Under \Cref{hyp:quantile:median},  for any $0 \leq  \kappa < d$, $\sup_{\param \in \Rset^d}  \int \norm{x-\param}^{-\kappa} \pi(x) \rmd x<
 \infty$.
\end{lemma}
\begin{proof}
  Let $0<\kappa < d$.
 \begin{align*}
   \int \norm{x-\param}^{-\kappa} \pi(x) \rmd x &= \int_0^{+\infty}  \rmd t \int_{\{x:  \norm{x-\param}^{\kappa} \leq 1/t \}} \pi(x) \rmd x \\
   & \leq 1 + \sup_{x \in \Rset^d} \pi(x) \ \int_{1}^{+\infty} \rmd t
   \sup_{\param \in \Rset^d} \int_{\{x: \norm{x-\param} \leq t^{-\kappa} \}}
   \rmd x  \\
& \leq 1  + C  \int_1^{+\infty}  t^{-d/\kappa}  \rmd t\eqsp,
 \end{align*}
 for a finite constant $C$, which does not depend on $\param$.
\end{proof}

\begin{proof}[\Cref{prop:example-quantile}]
As $\norm{\H{\param}{x}} = 1$, \Cref{hyp:PhiandH} is satisfied with the constant function $W = 1$.
By \cite[Lemma~19-(ii)]{arcones:1998}, there exists $C$ such that for any
$x,\param,\incr$ with $x \neq \param$, we have
\[
\norm{\norm{x-\param+\incr} - \norm{x -\param} + \pscal{\incr}{\frac{x-\param}{\norm{x-\param}}}} \leq C  \frac{\norm{\incr}^2}{\norm{x-\param}}
\]
Since $\sup_{\param \in \thetaset} \int \norm{x-\param}^{-1} \pi(x) \rmd x <
\infty$ (see \Cref{lem:ControlMoment} below), then this inequality implies that
$w$ is differentiable and $\nabla w(\param) = - \int (x-\param)/\norm{x-\param}
\pi(x) \rmd x$.  The dominated convergence theorem implies that $w$ is
continuously differentiable.
\Cref{hyp:LyapunovFunction}-\ref{hyp:LyapunovFunction:compacity} follows from
the lower bound $ w(\param) \geq \norm{ \param} - \int \norm{x} \pi(x) \rmd x$
and the continuity of $w$.  We have $\nabla w =-h$ from which
\Cref{hyp:LyapunovFunction}-\ref{hyp:LyapunovFunction:scal} trivially follows.
Finally, by \Cref{hyp:quantile:median} and \cite{milasevic:ducharme:1987},
$\calL$ contains a single point, and
\Cref{hyp:LyapunovFunction}-\ref{hyp:LyapunovFunction:calL} is satisfied.

Let $\param, \param' \in \thetaset$. For any $x \notin \{\param, \param' \}$,
  \begin{align*}
    \norm{\H{\param'}{x}-\H{\param}{x}} &=
    \norm{\frac{x-\param'}{\norm{x-\param'} \norm{x-\param}}
      \left(\norm{x-\param} - \norm{x-\param'}\right) + \frac{\param -
        \param'}{\norm{x-\param}}} \leq 2 \frac{\norm{ \param'-
        \param}}{\norm{x-\param}} \eqsp.
\end{align*}
Define $\calH_{\param,\delta}(x) \eqdef \sup_{\param' \in
  \mathcal{B}(\param,\delta)} \norm{\H{\param'}{x}-\H{\param}{x}}$.  Let
$0<\beta<1/d$. Then
\begin{align*}
  &  \int \pi(x) \calH_{\param,\delta} (x) \rmd x = \int_{x \in \mathcal{B}(\param,\delta+\delta^{\beta})} \pi(x) \calH_{\param,\delta}(x) \rmd x + \int_{x \notin \mathcal{B}(\param,\delta+\delta^{\beta})} \pi(x) \calH_{\param,\delta} (x) \rmd x \\
  & \leq 2 \sup_{x\ in \Rset^d} \pi(x) \, \int_{x \in \mathcal{B}(\param,\delta+\delta^{\beta})}  \rmd x + 2 \int_{x \notin \mathcal{B}(\param,\delta+\delta^{\beta})}  \sup_{\param' \in \mathcal{B}(\param,\delta)} \frac{\norm{ \param' - \param}}{\norm{x-\param}} \pi(x) \rmd x \\
  & \leq C \delta^{\beta d} + 4 \delta^{1-\beta} \eqsp,
\end{align*}
for a constant $C$ which is finite by \Cref{hyp:quantile:median}. Hence, and
\Cref{hyp:H:deltaH} is satisfied with $\alpha = (\beta d) \wedge (1-\beta) <
1$.  This holds true with $\beta =1/(1+d)$ for which $\beta d = 1-\beta$.
\end{proof}

\subsection{Proofs of \Cref{sec:vectquant}}
\label{proof:sec:vectquant}

We start with a preliminary lemma which gives a control on the intersection of
two Voronoi cells associated with  $\btheta, \bar{\btheta}
\in (\Rset^d)^N$.
\begin{lemma}
\label{lemme:inter:voronoi}
For any compact set $\calK$ of $\Theta$, there exists $\delta_{\calK}>0$ such
that for any $\btheta \in \calK$ and any $i \neq j$:
\begin{enumerate}[label=(\roman*)]
 \item  \label{lemme:inter:voronoi:angle}
 \begin{equation*} \sup_{\delta \leq
       \delta_K} \frac{1}{\sqrt{\delta}} \sup_{\bar{\btheta} \in \calB(\btheta,
       \delta)^N \cap \Theta} \norm{ \frac{\bar{\theta}^{(j)} - \bar{\theta}^{(i)}}{\norm{ \bar{\theta}^{(j)}
         - \bar{\theta}^{(i)} }} - \frac{\theta^{(j)} - \theta^{(i)}}{\norm{ \theta^{(j)} - \theta^{(i)}
         }} }< \infty \eqsp.
 \end{equation*}
\item \label{lemme:inter:voronoi:rect} for any $\delta \leq \delta_\calK$,
  there exists a measurable set $R_{i,j}(\btheta,\delta)$ such that
  \begin{align*}
  &\sup_{\bar{\btheta} \in \calB(\btheta, \delta) \cap \Theta} \indic{ C_i(\btheta) \cap C^{(j)}(\bar{\btheta}) \cap \calB(0,\Delta)} \leq \indic{R_{i,j}(\btheta,\delta)} \eqsp, \\
  &\sup_{\delta \leq \delta_K} \frac{1}{\sqrt{\delta}}  \int  \indic{R_{i,j}(\btheta,\delta)}(x) \,  \rmd x < \infty \eqsp.
  \end{align*}
 \end{enumerate}
\end{lemma}
\begin{proof}
  Let $\calK$ be a compact set of $\Theta$. The function on $(\Rset^d)^N$ given
  by $\btheta \mapsto \min_{i \neq j} \norm{ \theta^{(i)} - \theta^{(j)} }$ is
  continuous.  Since $\calK$ is a compact subset of $\Theta$, there exists
  $b_{\calK}>0$ such that for any $\btheta \in \calK$, $ \min_{i \neq j} \norm{
    \theta^{(i)} - \theta^{(j)} } \geq b_{\calK}$.  Choose $\delta_{\calK} \in (0,
  b_\calK /2 \wedge 1)$.  Let $i \neq j \in \{1, \cdots, N \}$ and $\btheta \in
  \calK$ be fixed.  For any $\delta \leq \delta_\calK$ and $\bar{\btheta} \in
  \calB(\btheta, \delta)$, it holds
\begin{align}
  \norm{ \bar{\theta}^{(j)} - \bar{\theta}^{(i)} } & \geq \norm{ \theta^{(j)} -
    \theta^{(i)}} - \norm{\bar{\theta}^{(j)} -
    \theta^{(j)} }- \norm{ \bar{\theta}^{(i)} -\theta^{(i)}}  \geq \norm{ \theta^{(j)} - \theta^{(i)}}- 2 \delta \label{eq:qv:minortheta'} \\
  & \geq b_\calK - 2 \delta >0 \eqsp. \nonumber
\end{align}
Similarly,
\begin{equation}
  \label{eq:qv:majortheta'}
   \norm{\bar{\theta}^{(j)} - \bar{\theta}^{(i)} } \leq \norm{ \theta^{(j)} - \theta^{(i)}} + 2 \delta \eqsp.
\end{equation}
Define $n = (\theta^{(j)} - \theta^{(i)})/\norm{ \theta^{(j)} - \theta^{(i)} }$ and  $n' = (\bar{\theta}^{(j)} - \bar{\theta}^{(i)})/\norm{ \bar{\theta}^{(j)} -\bar{\theta}^{(i)}}$.
\begin{enumerate}[label=(\roman*), wide=0pt, labelindent=\parindent]
\item We have $\norm{n - n'}^2 = 2 \left(1 - \pscal{n}{n'} \right)$. In
  addition, for any $\delta \leq \delta_\calK$ and $\bar{\btheta} \in \calB(\btheta,
  \delta)$,
\begin{align*}
  \pscal{n}{n'} &= \norm{\theta^{(j)} - \theta^{(i)} }^{-1} \pscal{\theta^{(j)}
    -\theta^{(i)}}{n'}  \\
  & = \norm{\theta^{(j)} - \theta^{(i)} }^{-1} \pscal{ \norm{
      \bar{\theta}^{(j)} -\bar{\theta}^{(i)} } n'+
    \theta^{(j)} - \bar{\theta}^{(j)} + \bar{\theta}^{(i)} - \theta^{(i)}}{n'}  \\
  & \geq \frac{\norm{\bar{\theta}^{(j)} - \bar{\theta}^{(i)}}}{ \norm{
      \theta^{(j)} - \theta^{(i)} }} - \frac{2 \delta}{\norm{\theta^{(j)} -
      \theta^{(i)} }} \geq 1 - \frac{4 \delta}{ \norm{\theta^{(j)} -
      \theta^{(i)}}} \geq 1 - \frac{4 \delta}{b_{\calK}} \eqsp,
\end{align*}
where we used (\ref{eq:qv:minortheta'}) in the last equation.
Therefore
\begin{equation}
\label{eq:qv:upperbound:cosinus}
\norm{n -n'}^2 \leq 8 \delta / b_{\calK} \eqsp,
\end{equation}
\item Let $x \in C^{(i)}(\btheta)$. We write $ x - \theta^{(i)} = \pscal{x-\theta^{(i)}}{n} n + m $ where $\pscal{m}{n} =0$.
Using $ \norm{x-\theta^{(i)}}^2 = \left|\pscal{x-\theta^{(i)}}{n} \right|^2 +
\norm{m }^2$ and
$ x - \theta^{(j)} = \pscal{x-\theta^{(i)}}{n} n - \norm{ \theta^{(i)} - \theta^{(j)}} n + m $
we get
$$
\norm{x - \theta^{(j)}}^2 = \left|\pscal{x-\theta^{(i)}}{n} - \norm{ \theta^{(i)} - \theta^{(j)}}
\right|^2 + \norm{m}^2 \eqsp.
$$
Since $x \in C^{(i)}(\btheta)$, $ \norm{ x - \theta^{(i)} }
\leq \norm{x - \theta^{(j)}}$ so that $\left|\pscal{x-\theta^{(i)}}{n} \right|^2 \leq
\left|\pscal{x-\theta^{(i)}}{n} - \norm{ \theta^{(i)} - \theta^{(j)}} \right|^2$.  This
implies that $\pscal{x-\theta^{(i)}}{n} \leq \norm{\theta^{(j)} - \theta^{(i)} }/2$.
Therefore,
\begin{align*}
 C^{(i)}(\btheta) \subset \left\{ x \in \Rset^d, \pscal{x - \theta^{(i)}}{n} \leq \frac{1}{2} \norm{ \theta^{(j)} - \theta^{(i)} } \right\} \eqsp.
\end{align*}
Let now $x \in C_{j}(\bar{\btheta}) \cap \calB(0,\Delta)$. Following the same lines as
above and using (\ref{eq:qv:majortheta'})
\begin{align}
\label{eq:qv:surCj}
\pscal{x -\bar{\theta}^{(j)}}{n'} \geq - \frac{1}{2} \norm{ \bar{\theta}^{(j)} - \bar{\theta}^{(i)}} \geq
-\frac{1}{2} \norm{\theta^{(j)} - \theta^{(i)} } - \delta\eqsp.
\end{align}
Moreover
\begin{align*}
  \pscal{x - \theta^{(i)}}{n}
  &= \pscal{x-\theta^{(i)}}{n-n'} + \pscal{x-\bar{\theta}^{(j)}}{n'} + \pscal{\bar{\theta}^{(j)}-\bar{\theta}^{(i)}}{n'} + \pscal{\bar{\theta}^{(i)}-\theta^{(i)}}{n'}  \\
  &= \pscal{x-\theta^{(i)}}{n-n'} + \pscal{x-\bar{\theta}^{(j)}}{n'} + \norm{\bar{\theta}^{(j)}-\bar{\theta}^{(i)}} + \pscal{\bar{\theta}^{(i)}-\theta^{(i)}}{n'} \eqsp.
\end{align*}
Since $x, \theta^{(i)} \in \calB(0, \Delta)$, we have by (\ref{eq:qv:minortheta'}),
(\ref{eq:qv:upperbound:cosinus}) and (\ref{eq:qv:surCj})
\begin{align*}
  \pscal{x - \theta^{(i)}}{n}
  & \geq - 2 \Delta \norm{n-n'} - \frac{1}{2} \norm{ \theta^{(j)} - \theta^{(i)}} - \delta + \norm{\theta^{(j)}-\theta^{(i)}} - 2 \delta -  \delta \\
  & \geq \frac{1}{2} \norm{ \theta^{(j)} - \theta^{(i)}} - 4 \delta -4 \Delta
  \sqrt{2/b_{\calK}} \sqrt{\delta} \eqsp.
\end{align*}
Therefore,
\begin{align*}
  C^{(j)}(\bar{\btheta}) \cap \calB(0, \Delta) \subset \left\{ x \in \Rset^d, \pscal{x -
      \theta^{(i)}}{n} \geq \frac{1}{2} \| \theta^{(j)} - \theta^{(i)} \| - 4 \delta -4 \Delta \sqrt{2/b_{\calK}} \sqrt{ \delta } \right\} \eqsp.
\end{align*}
Hence,
\begin{multline*}
C^{(i)}(\btheta) \cap C^{(j)}(\bar{\btheta}) \cap \calB(0, \Delta) \\
 \subset \left\{ x \in \calB(0,
    \Delta), \frac{1}{2} \norm{ \theta^{(j)} - \theta^{(i)} } - 4 \delta -4 \Delta \sqrt{2/b_{\calK}}
    \sqrt{\delta} \leq \pscal{x - \theta^{(i)}}{n} \leq \frac{1}{2} \norm{ \theta^{(j)} -
    \theta^{(i)} } \right\} \eqsp.
\end{multline*}
Finally, since $\delta_\calK < 1$, we have $\delta \leq \sqrt{\delta}$, and
this concludes the proof, by noticing that this last set is independent of
$\bar{\btheta}$.
\end{enumerate}
\end{proof}

\begin{proof}[Proof of \Cref{prop:qv:cv}]
  For any compact set $\calK \subset \thetaset$, there exists $C$ such that
  $\sup_{\btheta \in \calK} \norm{\H{\btheta}{u}} \leq C (\norm{u}+1)$.
  Therefore, \Cref{hyp:PhiandH} is satisfied with $W(u) = 1 + \norm{u}$. \\
  $w$ is nonnegative and continuously differentiable on $\thetaset$; since
  $\nabla w = - h$,
  \Cref{hyp:LyapunovFunction}-\ref{hyp:LyapunovFunction:scal} is satisfied. \\
  We now prove \Cref{hyp:LyapunovFunction}-\ref{hyp:LyapunovFunction:calL}; the
  proof is by contradiction.  Assume that $\calL$ is not included in a level
  set of $w$: then there exists a sequence $\{\btheta_q, q \geq 1 \}$ of
  $\calL$ such that $\lim_q w(\btheta_q) = + \infty$. Since $\widetilde w$ is
  bounded on $(\calB(0, \Delta))^N$, then $\lim_q \sum_{i \neq j}
  \norm{\theta^{(i)}_q - \theta^{(j)}_q}^{-2} = +\infty$ which implies that
  there exist a subsequence (still denoted $\{\theta_q, q \geq 1 \}$) and
  indices $i \neq j$ such that $\lim_q \norm{\theta_q^{(i)} - \theta_q^{(j)}} =
  0$. Since $\calL$ is closed, we proved that there exists a point $\lim_q
  \btheta_q$ in $\calL$ such that $\lim_q \theta_q^{(i)} = \lim_q
  \theta_q^{(j)}$. This is a
  contradiction since $\calL \subset \Theta$. \\
  Let us prove \Cref{hyp:H:deltaH}.  Let $\calK \subset \thetaset$ be a compact
  set.  We write
  \begin{align*}
    \norm{ \H{\btheta}{x} - \H{\bar{\btheta}}{x}} \leq \norm{ \tH{\btheta}{x} -
      \tH{\bar{\btheta}}{x}} + \lambda \sum_{i=1}^N \sum_{j \neq i}\left(
      \norm{\frac{\theta^{(i)} - \theta^{(j)}}{\norm{\theta^{(i)} -
            \theta^{(j)}}^4} - \frac{\bar \theta^{(i)} - \bar
          \theta^{(j)}}{\norm{\bar \theta^{(i)} - \bar \theta^{(j)}}^4}}
    \right) \eqsp.
  \end{align*}
  Since $\calK$ is a compact of $\Theta$, there exists a constant $C$ such
  that
\[
\norm{ \H{\btheta}{x} - \H{\bar{\btheta}}{x}} \leq \norm{ \tH{\btheta}{x} -
      \tH{\bar{\btheta}}{x}} + C \norm{\btheta - \bar{\btheta}} \eqsp.
\]
For any $\btheta,\bar{\btheta} \in \calK$ and any $x \in \Rset^d$,
\begin{multline*}
\norm{\tH{\bar \param}{x} - \tH{\param}{x}}^2/4 = \sum_{i=1}^N \left[ \norm{ \bar{\theta}_i - \theta^{(i)} }^2 \indic{C^{(i)}(\btheta) \cap C^{(i)}(\bar{\btheta})}(x) \right. \\
\left. + \norm{ \theta^{(i)} - x}^2 \indic{C_{i}(\btheta)
      \cap C^{(i)}(\bar{\btheta})^c}(x) + \norm{ \bar{\theta}^{(i)} - x}^2
    \indic{C^{(i)}(\btheta)^c \cap C^{(i)}(\bar{\btheta})}(x)\right] .
\end{multline*}
Therefore, for any $x \in \calB(0,\Delta)$, any $\btheta \in \calK$, and any
$\bar{\btheta} \in \calB(0,\delta)$,
\begin{align*}
  \norm{\tH{\bar{\btheta}}{x} - \tH{\btheta}{x}}/2 & \leq   \sqrt{\sum_{i=1}^N  \norm{ \bar{\theta}^{(i)} - \theta^{(i)} }^2}   +  \sum_{i=1}^N \sum_{j=1, j \neq i}^N  \norm{ \theta^{(i)} - x} \indic{C^{(i)}(\btheta) \cap C^{(j)}(\bar{\btheta})}(x) \\
  & \quad \phantom{\leq} +  \sum_{i=1}^N \sum_{j=1, j \neq i}^N  \norm{ \bar{\theta}^{(i)} - x} \indic{C^{(j)}(\btheta) \cap C^{(i)}(\bar{\btheta})}(x) \\
  & \leq \delta + 2 \Delta N^2 \sup \limits_{i \neq j} \indic{C^{(i)}(\btheta) \cap
    C^{(j)}(\bar{\btheta}) \cap \calB(0, \Delta)}(x) \eqsp.
\end{align*}
By \Cref{lemme:inter:voronoi}, there exists $\delta_{\calK}$ such that for
any $\delta \leq \delta_{\calK}$, there exist a measurable set
$R_{i,j}(\btheta,\delta)$ such that
\begin{align*}
 \sup \limits_{\bar{\btheta} \in \calB(0,\delta)} \indic{C^{(i)}(\btheta) \cap C^{(j)}(\bar{\btheta}) \cap \calB(0,\Delta)}(x) \leq \indic{R_{i,j}(\btheta,\delta)}(x) \eqsp.
\end{align*}
Therefore, $ \norm{\tH{\bar{\btheta}}{x} - \tH{\btheta}{x}}/2 \leq \delta + 2 \Delta N^2
\sup \limits_{i \neq j} \indic{R_{i,j}(\btheta,\delta)}(x)$.  Under
\Cref{hyp:qv}, $\pi$ is bounded on $\thetaset$. In addition,
\Cref{lemme:inter:voronoi} shows that
\begin{align*}
  \sup_{\delta \leq \delta_{\calK}} \frac{1}{\sqrt{\delta}} \sup_{\btheta \in
    \calK} \sup_{i \neq j} \int \indic{R_{i,j}(\btheta,\delta)}(x) \rmd x < \infty
  \eqsp.
\end{align*}
Then, there exists $C'$ such that for any $\delta \leq \delta_{\calK}$,
\begin{align*}
  \sup_{\btheta \in \mathcal{K}} \int \pi(\rmd x) \sup_{\{\bar{\btheta}, \norm{
      \bar{\btheta}- \btheta } \leq \delta \}} \norm{ \tH{\bar{\btheta}}{x} - \tH{\btheta}{x} }
  \leq C' \sqrt{\delta} \eqsp.
\end{align*}
Moreover, as $\sup_{\btheta \in \thetaset} \sup_{x \in \calB(0, \Delta)}
\norm{\tH{\btheta}{x}}< \infty$, for any $\delta \geq \delta_{\calK}$,
\begin{align*}
  \sup_{ \btheta \in \mathcal{K}} \int \pi(\rmd x) \sup_{\{\bar{\btheta}, \norm{
      \bar{\btheta}- \btheta } \leq \delta \}} \norm{\tH{\bar{\theta}}{x} - \tH{\btheta}{x} }
  \leq 2 \sup_{\Theta \times \supp(\pi)} \norm{\tH{\btheta}{x}}
  \frac{\sqrt{\delta}}{\min(1,\sqrt{\delta_{\calK}})} \eqsp.
\end{align*}
Therefore \Cref{hyp:H:deltaH} is satisfied with $\alpha =1/2$.
\end{proof}